\documentclass{article}

\usepackage[utf8]{inputenc}
\usepackage[T1]{fontenc}
\usepackage{amsmath,amssymb,amsthm,bbm,float,graphicx,geometry,lmodern,mathtools,parskip,setspace,subcaption}
\usepackage[colorlinks, linkcolor = blue!80!black, citecolor = blue!80!black,breaklinks, pdfauthor={Emmanuel Jacob, Benedikt Jahnel, Lukas Luechtrath}]{hyperref}
\usepackage{mathrsfs}
\usepackage{enumerate}
\usepackage{nicefrac}
\usepackage{todonotes}
\usepackage{comment}
\usepackage{titling}
\usepackage{orcidlink}

\usepackage{tikz}
\usetikzlibrary{backgrounds}
\usetikzlibrary{patterns}
\usetikzlibrary{positioning, shapes.geometric}
\usetikzlibrary{external, arrows, decorations.pathmorphing} 

\usepackage{caption}
\usepackage{graphicx}
\graphicspath{{Images/}}
\allowdisplaybreaks
\usepackage[raggedright]{titlesec} 
\usepackage{booktabs} 

\newcommand{\N}{\mathbb{N}}

\newtheorem{theorem}{Theorem}[section]
\newtheorem{lemma}[theorem]{Lemma}
\newtheorem{corollary}[theorem]{Corollary}
\newtheorem{prop}[theorem]{Proposition}

\newtheorem{definition}[theorem]{Definition}
\newtheorem{remark}[theorem]{Remark}

\usepackage{nameref}

\makeatletter
\let\orgdescriptionlabel\descriptionlabel
\renewcommand*{\descriptionlabel}[1]{%
  \let\orglabel\label
  \let\label\@gobble
  \phantomsection
  \edef\@currentlabel{#1}%
  \let\label\orglabel
  \orgdescriptionlabel{#1}%
}

\renewcommand{\P}{\mathbb{P}}
\newcommand{\x}{\boldsymbol{x}}
\newcommand{\y}{\boldsymbol{y}}
\newcommand{\1}{\mathbbm{1}}
\newcommand{\G}{\mathscr{G}}
\newcommand{\scrG}{\G}
\newcommand{\scrV}{\mathscr{V}}
\newcommand{\scrE}{\mathscr{E}}

\newcommand{\X}{\boldsymbol{X}}
\newcommand{\Y}{\boldsymbol{Y}}
\newcommand{\E}{\mathbb{E}}
\newcommand{\0}{\mathbf{o}}
\renewcommand{\o}{\0}
\newcommand{\R}{\mathbb{R}}
\renewcommand{\d}{\mathrm{d}}

\newcommand{\Z}{\mathbb{Z}}

\newcommand{\cG}{\mathcal{G}}
\newcommand{\scrM}{\mathscr{M}}

\newcommand{\p}{\mathbf{p}}
\newcommand{\deff}{\delta_\textup{eff}}
\newcommand{\cL}{\mathcal{L}}
\newcommand{\cF}{\mathcal{F}}
\newcommand{\cE}{\mathcal{E}}
\newcommand{\scrL}{\mathscr{L}}
\newcommand{\wlc}{\widehat{\lambda}_{\textsf{c}}}
\newcommand{\lc}{\lambda_{\textsf{c}}}

\newcommand{\eps}{\varepsilon}

\pretitle{\centering\LARGE\scshape}
 \posttitle{\vskip 0.75cm}

 \predate{\vskip 0.75 cm \centering\large}
 \postdate{\par}

\usepackage[style = numeric, sorting=nyt, url = false, abbreviate=false, maxbibnames=9, sortcites=true, doi = true, backend = biber, giveninits = true, isbn=false]{biblatex}
\renewbibmacro{in:}{\ifentrytype{article}{}{\printtext{\bibstring{in}\intitlepunct}}}
\bibliography{bib.bib}

\title{Subcritical annulus crossing in spatial random graphs}
\thanksmarkseries{arabic}

\author{
Emmanuel Jacob\thanks{\'Ecole Normale Sup\'erieure de Lyon, 46 all\'ee d'Italie, 69007 Lyon, France} \\ emmanuel.jacob@ens-lyon.fr
\and
Benedikt Jahnel \orcidlink{0000-0002-4212-0065}\thanks{Technische Universit\"at Braunschweig, Universit\"atsplatz 2, 38106 Braunschweig, Germany}\thanksgap{0.4ex} \thanks{Weierstrass Institute for Applied Analysis and Stochastics, Anton-Wilhelm-Amo-Str.\ 39, 10117 Berlin, Germany} \\ benedikt.jahnel@tu-braunschweig.de
\and
Lukas L\"{u}chtrath \orcidlink{0000-0003-4969-806X}\thanksmark{2}\\lukas.luechtrath@wias-berlin.de \\
}

\date{March 16, 2026}

\begin{document}

\maketitle

\begin{spacing}{0.9}
\begin{abstract} 
 \noindent We consider general continuum percolation models obeying sparseness, translation invariance, and spatial decorrelation. In particular, this includes models constructed on general point sets other than the standard Poisson point process or the Bernoulli-percolated lattice. Moreover, in our setting the existence of an edge may depend not only on the two end vertices but also on a surrounding vertex set and models are included that are not monotone in some of their parameters. 
 We study the critical \emph{annulus-crossing} intensity \(\widehat{\lambda}_{\rm c}\), which is smaller or equal to the classical critical percolation intensity \(\lambda_{\rm c}\) and derive a condition for \(\widehat{\lambda}_{\rm c}>0\) by relating the crossing of annuli to the occurrence of long edges. This condition is sharp for models that have a modicum of independence. In a nutshell, our result states that annuli are either not crossed for small intensities or crossed by a single edge. Our proof rests on a multiscale argument that further allows us to directly describe the decay of the annulus-crossing probability with the decay of long edges probabilities. 

 \smallskip
 
 \noindent We apply our result to a number of examples from the literature. Most importantly, we extensively discuss the \emph{weight-dependent random connection model} in a generalised version, for which we derive sufficient conditions for the presence or absence of long edges that are typically easy to check. These conditions are built on a decay coefficient \(\zeta\) that has recently seen some attention due to its importance for various proofs of global graph properties.

 \medskip
\noindent\footnotesize{{\textbf{AMS-MSC 2020}: Primary: 60K35; Secondary: 90B15, 05C80}

}
\end{abstract}
\end{spacing}

\pagebreak

\tableofcontents

\section{Introduction and main results}\label{sec:intro}
The standard objects studied in continuum percolation theory are random graphs \(\G\) on the points of a homogeneous Poisson point process on \(\R^d\) of intensity \(\lambda>0\). The spatial embedding of the vertices enters the connection probability in a way that spatially close vertices are connected by an edge with a higher probability than those farther apart. This is typically done in a monotone manner, i.e., increasing \(\lambda\) leads to more edges on average.

The central question in percolation theory is whether there exists a non-trivial critical Poisson intensity \(\lc\in(0,\infty)\) such that the connected component or \emph{cluster} of the origin (added to the graph if necessary) is infinite with a positive probability for all \(\lambda>\lc\), but is finite almost surely for \(\lambda<\lc\). We call the regime \((0,\lc)\) the \emph{subcritical (percolation) phase} and \((\lc,\infty)\) the \emph{supercritical (percolation) phase}. The ergodicity of the model ensures the almost-sure existence of an infinite cluster in \(\G\) if \(\lambda>\lc\) and the almost-sure absence of such a cluster if \(\lambda<\lc\), respectively. Further, an existing infinite cluster is typically (but not always) unique~\cite{BurtonKeane89,chebunin_Last_2024_uniqueness}. 

In dimensions \(d\geq 2\) percolation models often contain a supercritical phase~\cite{MeesterRoy1996} and proving a non-trivial phase transition \(\lc\in(0,\infty)\) reduces to showing \(\lc>0\). As an example, consider the sparse Boolean model in which each vertex is the centre of a ball with an independent and identically distributed volume of finite expectation. Each pair of vertices is connected by an edge if their associated balls intersect. In order to show \(\lc>0\) for this model, Gou\'{e}r\'{e} introduced the alternative critical \emph{annulus-crossing} intensity \(\wlc\)~\cite{Gouere08}, which is the smallest \(\lambda\) for which the probability of finding a path starting in the centred ball \(B(r)\) and ending outside the ball \(B(2r)\) is bounded from zero for all \(r\). Put differently, for \(\lambda>\wlc\), annuli of radii \(r\) and \(2r\) are crossed. By monotonicity \(\wlc\leq \lc\), and the existence of a subcritical annulus-crossing phase implies a subcritical percolation phase. Despite the elegance of the proof of \(\wlc>0\), another advantage in looking at annuli is that the method immediately quantifies the decay of the annulus-crossing probability and further deduces the tail of the Euclidean diameter of a typical subcritical cluster. In recent works, the equality of \(\wlc\) and \(\lc\) for the Boolean model has been studied~\cite{AhlbergTassionTeixeira2018,DCopinRaoufiTassion2020,DembinTassion2022} and \(\wlc=\lc\) was shown with some technical exceptions if the balls are generated by Pareto distributed radii~\cite{DembinTassion2022}.

We build on the work in~\cite{Gouere08} and study annulus crossings for general spatial random graphs on the vertices of stationary point processes. We show that \(\wlc>0\) whenever correlations are weak at large distances and the occurrence of certain long edges is unlikely for small intensities. On the contrary, if long edges are likely for all intensities then \(\wlc=0\). Our results show that the existence of a subcritical annulus-crossing phase is strongly related to the long-range properties of a model. In fact, strong long-range effects can remove a subcritical annulus-crossing phase entirely. Indeed, in many situations, either annuli are crossed with a single edge or are not crossed at all for small intensities. Put differently, \(\wlc=0\) is often equivalent to the presence of `long edges'. However, the long edges alone can typically not remove the subcritical percolation phase. Unlike in the Boolean model, in models with many long edges both intensities are therefore not always the same. More precisely, some of the models we consider have
\[0=\wlc<\lc,\]
however it is unclear whether any of our models could feature the property $0<\wlc<\lc$. Let us mention that we do not require the graphs to be monotone in \(\lambda\) and \(\wlc\) may not be the unique phase transition. However, \(\wlc>0\) still implies \(\lc>0\) and our results give an easy condition for the existence of a subcritical phase. 

We apply our results to an extensive list of examples from the literature. The most important model class we discuss is a generalisation of the \emph{weight-dependent random connection model}~\cite{GHMM2022}. This class contains many well-established models with heavy-tailed degree distributions and/or long-range interactions, and is an active field of research. We derive `easy checkable' conditions for \(\wlc>0\) and \(\wlc=0\), respectively. These conditions rest on quantifying, on a polynomial scale, the amount of vertices located in large balls being incident to edges equally long as the ball's radius. Although slightly differently derived, as we allow for more correlated models, the found exponent coincides with the one introduced in~\cite{JorritsmaKomjathyMitsche2023,jorritsmaKomjathyMitsche2023_LRP,jorritsmaKomjathyMitsche2024LDP} to describe the cluster-size decay in large but finite supercritical components when many long edges are present. The exponent extends and specifies the notion of the \emph{effective decay exponent} introduced~\cite{GraLuMo2022,Moench2024}, which makes a model comparable to classical long-range percolation~\cite{Schulman1983} and has for instance been used to identify the existence of supercritical percolation phases in one dimension~\cite{GraLuMo2022} or to prove continuity of the percolation function in regimes with many long edges~\cite{Moench2024}. In conclusion, all these works combined underline the significance of long-edges and long-range interactions on a graphs connectivity properties.   

We explain our general framework in Section~\ref{sec:framework} and formulate our main results in Section~\ref{sec:mainResults}. We discuss the weight-dependent random connection model in detail in Section~\ref{sec:wdrcm}. In Section~\ref{sec:correlatedExamples} we exemplary discuss additional and more correlated examples. Finally, we give the proofs of our results in Section~\ref{sec:proofs}. Throughout the manuscript, we use the standard Landau notation as well as the convention that \(f\asymp g\) stands for \(f=\Theta(g)\). We further denote by \(\sharp A\) the cardinality of a countable set \(A\) and by \(|x|\) the Euclidean norm of \(x\in\R^d\). Additional notation used in the manuscript is introduced where it is needed first.  
 
 \subsection{Framework} \label{sec:framework}
 In this section, we define the basic properties of all considered graphs. We aim to study countably infinite (undirected) random graphs embedded into \(\R^d\) in a wide spectrum that not only contains the classical models of continuum percolation but models from other contexts as well. More precisely, we consider a graph \(\scrG=(\scrV,\scrE)\) whose vertices are given via a simple, stationary point process on \(\R^d\) of intensity \(\lambda>0\). Edges are drawn with probabilities depending on the vertices locations, the point cloud around the vertices and additional features of the vertices. Let us first give the basic properties of \(\scrG\) that we will always assume for all models throughout the manuscript without further mentioning them.  
 
 \begin{description}
 	\item[(G1)\label{G:Point_process}] 
 		The locations of the vertices in \(\scrV\) are given by a \emph{locally finite, stationary, and simple point process} on \(\R^d\) of intensity \(\lambda>0\). That is, the intensity measure is the Lebesgue measure multiplied by \(\lambda\)~\cite{LastPenrose2017}. Note that this refers to the location of the vertices only and the vertices may carry additional marks or weights. In fact, in most examples, vertex marks are needed to model a vertex' attraction or influence in some way. Throughout the manuscript, we denote vertices by \(\x\in\scrV\). For a vertex \(\x\in\scrV\) we denote its location by \(x\in\R^d\). Although \(\scrV\) refers to the whole vertex set and may contain additional markings, we still write \(\x\in\scrV\cap A\) for a vertex \(\x\) with location \(x\in A\subset\R^d\).  
 	\item[(G2)\label{G:Translation}] 
 		We assume that \(\G\) is \emph{translation invariant}. That is, \(\G\) and \(\G+x\) have the same distribution for each \(x\in\R^d\), where \(\G+x\) is the graph constructed on the points of \(\scrV+x\). That is, the connection mechanism does not change if the location of each vertex is shifted by \(x\). 	
 	\item[(G3)\label{G:locallyFinite}] 
 		The graph \(\scrG\) is \emph{locally finite}, i.e.\ all vertices have finite degree, almost surely. 
 \end{description}
 
 We denote the underlying probability measure by \(\P_\lambda\) and write \(\E_\lambda\) for the associated expectation operator.
 
 \begin{remark} \label{rem:Assumptions} ~\
 	\begin{enumerate}[(i)]
 		\item 
 			Our setup can easily be adapted to vertex sets based on infinite subsets of \(\Z^d\). In that case, translation invariance is with respect to shifts \(x\in\Z^d\) and the subset is to be generated via some stationary site-percolation process. Here, the intensity is simply given as \(\lambda:=\P(o\in \scrV)\). The standard example is Bernoulli site-percolation with retention parameter \(\lambda\in(0,1]\) \cite{Grimmett1999}. We will discuss other and more correlated examples in Section~\ref{sec:correlatedExamples}. We only require one technical assumptions for our proof to work. Instead of working with the Euclidean norm, we have to work under the norm \(||\cdot||_\infty\) on the lattice in order to guarantee appropriate scalings of annulus boundaries. However, due to equivalence of norms, this causes no lack of generality. To keep notation concise, we will only use the point process notation and the Euclidean norm throughout the manuscript. We comment in Remark~\ref{rem:proofMain} after the proof of our main theorem on the changes that have to be made when replacing the point process with a percolated lattice. 
 		\item 
 			The assumption that the graph is locally finite, i.e.\ Property~\ref{G:locallyFinite}, is not strictly necessary to formulate our results. Indeed, our main assumptions concerning long-range effects (cf.\ Definitions~\ref{def:mixing} and~\ref{def:longEdges} below) imply Property~\ref{G:locallyFinite} whenever a subcritical annulus-crossing phase exists. Nonetheless, we retain local finiteness as a standing assumption, both for clarity of exposition and because our motivation and results are primarily meaningful in this setting. We also note that local finiteness of both, the graph and the underlying vertex set, entails that connection probabilities decay with distance: vertices that are close together are, in general, more likely to be connected than an otherwise comparable pair separated by a much larger distance. This decay, however, need not be monotone. 		
 	\end{enumerate}   	
\end{remark}
 	
Another graph property satisfied by many models in the literature is \emph{monotonicity} in \(\lambda\). Although, we already mentioned that not all of the considered models are monotone, we shall see that the monotonicity property allows for stronger statements. We call an event \(\mathcal{E}\) \emph{increasing} if for all realisations \(\omega\in\mathcal{E}\) and any other realisation \(\omega'\) with \(\scrV(\omega)\subset \scrV(\omega')\) and \(\scrE(\omega)\subset\scrE(\omega')\), we have \(\omega'\in\mathcal{E}\). Here, \(\scrV(\omega)\) denotes the vertex set and \(\scrE(\omega)\) the edge set of \(\omega\). 
 		We say that \(\scrG\) is \emph{monotone} if 
 	\begin{equation*} \tag{Mon}\label{G:monotone}
 		\text{ for all }\lambda<\lambda' \text{ and any increasing event }\mathcal{E}\text{, we have } \P_\lambda(\mathcal{E})\leq \P_{\lambda'}(\mathcal{E}).
 	\end{equation*} 	
 
 \paragraph{Properties related to long-range effects.}	
 Having dealt with the more standard properties, we next discuss more specific properties subject to long-range effects of the graph. This aspect is crucial for our main result relating the presence or absence of long edges to the non-existence or existence of a subcritical annulus-crossing phase. Let us first specify some notation and properly introduce the critical annulus-crossing intensity. To this end, denote by \(B(x,r)\) the Euclidean ball of radius \(r\) centred at \(x\), where \(B(r)=B(o,r)\) represents the ball centred at the origin $o$. Denote further by \(B(x,r)^\textsf{c}=\R^d\setminus B(x,r)\) the complement of a ball and let \(\x\sim \y\) be the event that vertices in \(\scrV\) located at \(x\) and \(y\) are connected by an edge in \(\scrG\). The event that \(\x\) and \(\y\) are connected by a self-avoiding path is denoted by \(\x\xleftrightarrow[]{} \y\). Similarly, 
 	\[
 		A\xleftrightarrow[]{} B \quad :\Leftrightarrow \quad  \exists \x\in A\cap \scrV, \y\in B\cap\scrV\colon \x\xleftrightarrow[]{}\y
 	\] 
 	is the event that two measurable sets \(A\) and \(B\) are connected by a path, meaning there exists a vertex in each respective set that lies on the same path. We define the \emph{critical annulus-crossing intensity} by 
 	\begin{equation*}
 		\wlc = \inf\big\{\lambda>0\colon \limsup_{r\to\infty}\P_{\lambda}\big(B(r)\xleftrightarrow[]{}B(2r)^\textsf{c}\big)>0\big\}.
 	\end{equation*}
	If \(\wlc>0\), the defining probability converges to zero for all \(\lambda<\wlc\). However, if the graph is \emph{not monotone}, this does not imply that the limit of said probability is bounded from zero for all \(\lambda>\wlc\). In fact, there could in principle be another phase \((\lambda_1,\lambda_2)\) such that \(\wlc<\lambda_1<\lambda_2\leq\infty\) and \(\P_{\lambda}\big(B(r)\xleftrightarrow[]{}B(2r)^\textsf{c}\big)\to 0\) for all \(\lambda\in(\lambda_1,\lambda_2)\). Monotonicity in the sense of~\eqref{G:monotone} on the other hand clearly guarantees \(\wlc\) to be the unique such critical point.

	Finally, let us introduce the main properties of \(\scrG\) that quantify the long-range interactions. Typically, there are two ways in which these interactions can arise. One is through the presence of long edges in the graph that connect vertices far apart from each other. The other is via correlations of the local configurations of the graph across distant regions. Naturally, both factors can strongly influence the crossings of certain annuli. While the main scope of this manuscript lies on the effects of long edges, we also deal with correlations of the latter type for full generality. More precisely, let us denote by \(\scrG\cap B(x,r)\) the subgraph of \(\scrG\) induced by the vertices located in \(B(x,r)\), and define the \emph{local annulus-crossing} event by 
	\begin{equation}\label{eq:local_annulus_cross}
		\cG(x,r) := \big\{B(x,r)\xleftrightarrow[]{} B(x,2r)^\textsf{c} \text{ in }\scrG\cap B(x,3r)\big\}.
	\end{equation}
	In other words, the annulus with inner and outer radii \(r\) and \(2r\) centred in \(x\) is crossed by the graph using only vertices and edges within \(B(x,3r)\). We abbreviate \(\cG(o,r)=\cG(r)\). For our purpose, we require the correlations between two such local annulus-crossings to be sufficiently weak at greater distances.
    
    \begin{definition}[Mixing] \label{def:mixing} 
		We say that \(\scrG\) is \emph{mixing} for \(\lambda>0\) if, for all \(|x|\ge 10\), we have 
  		\begin{equation*}\tag{\(\scrM\)}\label{G:lambda-mix}
  		 	\lim_{r\to\infty}\big|\operatorname{Cov}_\lambda\big(\1_{\cG(r)}, \1_{\cG(xr,r)}\big)\big| = 0.
  		\end{equation*}
  		We denote this property by \(\scrM_\lambda\). We say that the graph is \emph{uniformly mixing} for \(\lambda>0\), denoted by \(\overline{\scrM}_\lambda\), if, for any \(|x|\ge 10\), we have
  		\begin{equation*}\tag{\(\overline{\scrM}\)}\label{G:mixing}	            \lim_{r\to\infty}\sup_{\lambda'<\lambda}\big|\operatorname{Cov}_{\lambda'}\big(\1_{\cG(r)}, \1_{\cG(xr,r)}\big)\big| = 0.
  		\end{equation*}
  		In some situations we require a polynomial decay of the correlations. If this is the case, we say that \(\scrG\) is \emph{polynomially mixing} for \(\lambda>0\) with exponent \(\xi=\xi_\lambda<0\) if, for any \(|x|\ge 10\), we have
  		\begin{equation*}\tag{\({{\mathscr{P}}\!\!{\mathscr{M}}}\)}\label{G:P_mixing}
 		\limsup_{r\to\infty}\frac{\log\big|\operatorname{Cov}_{\lambda}\big(\1_{\cG(r)}, \1_{\cG(xr,r)}\big)\big|}{\log r}\leq \xi
 		\end{equation*}
 		and denote this property by \({{\mathscr{P}}\!\!{\mathscr{M}}}_\lambda^{\xi}\). 
	\end{definition}
    Let us mention that \(|x|\geq 10\) in the definition ensures that the balls of radius \(3r\) subject to the two local annulus crossing events are disjoint and of order \(r\) apart. This is the important property for our proof to work and the particular choice \(|x|\geq 10\) could be replaced by any other bound guaranteeing it. 
    
    Many important models from the literature actually have the following independence assumptions in their construction, which immediately yields~\eqref{G:monotone} as well as~\ref{G:mixing}\(_\lambda\) and~\ref{G:P_mixing}\(_\lambda^\xi\) for all $\lambda>0$ and $\xi\ge -\infty.$

    \begin{definition}[Independent setting]\label{def:IndS} 
		We call the following setup the \emph{independent setting}: 
		Let the underlying point process have independent increments and is thus either given by a Poisson point process, cf.~\cite[Thm.~6.12]{LastPenrose2017}, or by a Bernoulli site-percolated lattice. Further, for a given collection of mutually different vertex pairs \(\{\x_{1_1},\x_{1_2}\},\dots,\{\x_{k_1},\x_{k_2}\}\), we have 
	\begin{equation*}\tag{Ind}\label{eq:independentEdges}
\P_\lambda^{\x_{1_1},\x_{1_2},\dots,\x_{k_2}}\Big(\bigcap_{j=1}^k \{\x_{i_1}\sim \x_{i_2}\} \, \Big| \, \scrV \Big) = \prod_{i=1}^k \mathbf{p}(\x_{i_1},\x_{i_2}), 
	\end{equation*} 
	for a suitable, symmetric function \(\mathbf{p}\). Here, \(\P_\lambda^{\x_{1_1},\x_{1_2},\dots,\x_{k_2}}\) denotes the law where the given points \(\x_{1_1},\dots,\x_{k_2}\) have been added to the graph in the Poisson point process case, or the law conditioned on the event that said vertices survived the site percolation, respectively. 
	\end{definition}

    We stress that the above definition can be seen as combining three independence properties:
    \begin{itemize}
        \item The independent increments property of the vertex set itself.
        \item Given the vertex set, the independence of the occurrence of the edges \(\{\x_{i_1}\sim \x_{i_2}\}\) for every unordered pair of vertices \(\{\x_{i_1},\x_{i_2}\}\).
        \item Given an unordered pair of vertices \(\{\x_{i_1},\x_{i_2}\}\) in the vertex set, the independence of the probability \(\mathbf{p}(\x_{i_1},\x_{i_2})\) of observing the edge from the surrounding vertex set.
    \end{itemize}

    The independent setting implies a natural coupling between the graph \(\scrG_\lambda\) and \(\scrG_{\lambda'}\) for $\lambda'<\lambda,$ where \(\scrG_{\lambda'}\) is obtained from \(\scrG_\lambda\) by performing a site percolation with retention parameter $\lambda'/\lambda$. In other words, we obtain \(\scrG_{\lambda'}\) by keeping each vertex of \(\scrG_\lambda\) independently with probability $\lambda'/\lambda,$ and further keeping all existing edges between the remaining vertices. In this coupling \(\scrG_{\lambda'}\) is a subgraph of \(\scrG_{\lambda}\), and we thus immediately obtain that the graph model is monotone.

    The independent setting also implies both~\ref{G:mixing}\(_\lambda\) and~\ref{G:P_mixing}\(_\lambda^{-\infty}\) for all \(\lambda>0\), as for \(|x|>2\) the local annulus crossing events \(\cG(r)\) and \(\cG(xr,r)\) are then independent and thus uncorrelated.

    However, the independent setting does not allow for models on correlated vertex sets, cf.\ Sections~\ref{sec:Cox} and~\ref{sec:FRI}, or models where the presence of an edge depends on vertices surrounding its end vertices, cf.~Section~\ref{sec:BoolInterferences}. In all these cases our results rely on the mixing assumptions~\ref{G:mixing}\(_\lambda\) or~\ref{G:P_mixing}\(_\lambda^\xi\).  
    	
	Let us finally introduce properly the notion of long edges and their occurrence. For the scope of the manuscript this is arguably the most interesting property as its influence on \(\wlc\) is studied in greatest detail. 
		
	\begin{definition} \label{def:longEdges}
	We define for all \(r>0\) and \(c>0\) the \emph{presence of long edges} event as
		\begin{equation}\label{eq:longEdges}
			\mathcal{L}(r,c):=\big\{\exists\ \x\sim \y\colon |x|<r, |x-y|>cr \big\}.
		\end{equation} 
		That is, we can find an edge of length at least \(cr\) in \(\scrG\) with one endpoint being located in \(B(r)\). We say that \(\scrG\) has 
		\begin{equation*}\tag{\(\scrL\)}\label{G:noLongEdge}
			\text{Property } \mathscr{L}_\lambda \quad \text{ if } \quad \limsup_{r\to\infty}\P_\lambda(\mathcal{L}(r,c))=0 \quad \text{ for all }c>0.  
		\end{equation*}
		If Property \(\scrL_{\lambda'}\) is \emph{uniformly} fulfilled for all \(\lambda'\leq \lambda\), we say that the graph has Property \(\overline{\scrL}_\lambda\). More precisely, \(\scrG\) has 
		\begin{equation*}\tag{\(\overline{\scrL}\)} \label{G:noLongEdgeUnif}
			\text{Property } \overline{\scrL}_\lambda \quad \text{ if }\quad \limsup_{r\to\infty}\sup_{\lambda'\leq \lambda}\P_{\lambda'}(\mathcal{L}(r,c))=0 \quad \text{ for all }c>0.
		\end{equation*}
	\end{definition}

	The Property~\ref{G:noLongEdge}\(_\lambda\) essentially says that finding long edges, i.e., edges of comparable length to the radius of the ball in which we search, is unlikely. Property~\ref{G:noLongEdgeUnif}\(_\lambda\) says this remains true for all smaller intensities. It can therefore be seen as a weak form of monotonicity (for small intensities). The following simple properties make this heuristic description actually quite pertinent. 
	\begin{lemma}[Properties of \ref{G:noLongEdge}\(_\lambda\)] \label{lem:PropL} ~\
		\begin{enumerate}[(i)]
			\item If the monotonicity assumption~\eqref{G:monotone} is fulfilled, then \ref{G:noLongEdge}\(_\lambda\) implies \ref{G:noLongEdgeUnif}\(_\lambda\).
			\item If \(\limsup_{r\to \infty} \P_{\lambda}(\cL(r,c))=0\) for some \(c>0\), then Property~\ref{G:noLongEdge}\(_\lambda\) holds. 
			\item If \(\limsup_{r\to \infty} \sup_{\lambda'<\lambda} \P_{\lambda}(\cL(r,c))=0\) for some \(c>0\), then Property~\ref{G:noLongEdgeUnif}\(_\lambda\) holds. 
		\end{enumerate}
	\end{lemma} 
	\begin{proof}
		The first part of the claim is immediate. The second and the third part can be proven simultaneously. First note that, for all \(c'>c\), we clearly have \(\mathcal{L}(r,c')\subset\mathcal{L}(r,c)\) and thus
		 \(
		 	\P_\lambda(\cL(r,c'))\leq \P_\lambda(\cL(r,c))
		 \)  
		 for all \(\lambda\). Consider now \(c'<c\). We can cover the ball \(B(r)\) with finitely many, say \(C\), balls of radius \(\varepsilon r\), where \(\varepsilon=c'/c\). Note that \(C\) does not depend on \(r\). Using a union bound and translation invariance~\ref{G:Translation} yields
		\[
			\P_\lambda(\mathcal{L}(r,c'))\leq C \, \P_\lambda(\exists \ \x\sim \y: |x|<\varepsilon r, |x-y|>c' r) = C \, \P_\lambda (\mathcal{L}(\varepsilon r, c)),
		\]
		for all \(\lambda\). This concludes the proof.
	\end{proof}  

	An important consequence of the claim is that one only has to check the limit of \(\P_{\lambda}(\cL(r,1))\) in order to decide whether the Properties~\ref{G:noLongEdge}\(_\lambda\) and~\ref{G:noLongEdgeUnif}\(_\lambda\) hold or not.

	\subsection{Main results}\label{sec:mainResults}
	Having established the framework and all necessary properties, we have now everything in place to formulate our main results relating Property~\ref{G:noLongEdgeUnif}\(_\lambda\) with the positivity of \(\wlc\) for a mixing graph.
	
	\pagebreak[3]
	
	\begin{theorem}[Existence of subcritical annulus-crossing phase]\label{thm:main} ~\
		\begin{enumerate}[(i)]
			\item If the uniform mixing Property~\ref{G:mixing}\(_\lambda\) and the no-long-edges Property~\ref{G:noLongEdgeUnif}\(_\lambda\) hold for some \(\lambda>0\), then \(\wlc>0\).
			\item If on the contrary Property~\ref{G:noLongEdge}\(_\lambda\) does not hold for any \(\lambda>0\), then \(\wlc=0\).
		\end{enumerate}	
	\end{theorem} 
	
	Our result becomes strongest for models in the independent setting.  
	
	\begin{theorem}\label{thm:equiStatement}
        Assume the independent setting in the sense of Definition~\ref{def:IndS}. Then we have \(\P_1(\cL(r,1))\asymp \P_\lambda(\cL(r,1))\) for all \(\lambda>0\). Further, Property~\ref{G:noLongEdge}\(_\lambda\) does not depend on \(\lambda\) and \ref{G:noLongEdge}\(_1\) is equivalent to \(\wlc>0\).
	\end{theorem}
 
	The above theorem essentially states that, in the independent setting, 
    each annulus is either crossed by a single edge or it is always possible to prevent the crossing by sufficiently reducing the intensity. While we have equivalence of \ref{G:noLongEdge}\(_1\) and \(\wlc>0\) in that case, we present an example in Section~\ref{sec:kNN} demonstrating that, in general, we cannot weaken either of the Properties~\ref{G:noLongEdgeUnif}\(_\lambda\) or~\ref{G:mixing}\(_\lambda\) in Part~(i) of Theorem~\ref{thm:main}. 
		
	Having established \(\wlc>0\) in the latter two theorems, we now have a closer look on the decay of the probability of the annulus-crossing event within the subcritical phase. 

	\begin{theorem} \label{thm:Diamter}
		Assume \(\wlc>0\), and that, for some \(\lambda<\wlc\), we have the \emph{polynomially mixing} Property~\ref{G:P_mixing}\(_{\lambda}^\xi\) for some \(\xi=\xi_\lambda<0\), as well as the existence of some $\zeta=\zeta_\lambda<0$ such that	
		\[
			\limsup_{r\to\infty}\frac{\log \P_\lambda(\cL(r,1))}{\log r}\le \zeta.
		\] 
		Then, we have 
		\[
			\limsup_{r\to\infty}\frac{\log \P_{\lambda}\big(B(r)\xleftrightarrow[]{}B(2r)^{\textsf{c}}\big)}{\log r}\leq \zeta \vee \xi.
		\]
	\end{theorem} 

	The theorem states in words that (under polynomial mixing and polynomial decay of \(\P_\lambda(\cL(r,1))\)) the annulus-crossing probability, defining \(\wlc\), decays at most polynomially, determined by either the rate of mixing or the rate of decay of the long-edges probability. This particularly implies that the probability of an annulus crossing and the probability of finding long edges both decay at the same speed at least whenever \(\P_\lambda(\cL(r,1))\) decays slower than \(r^{\xi}\). This is, for instance, always the case in the independent setting. 
		
	Furthermore, we note that the annulus-crossing event is closely related to the \emph{Euclidean diameter} of the components of the vertices located in \(B(r)\). Indeed, if the annulus crossing \(B(r)\xleftrightarrow[]{}B(2r)^\textsf{c}\) occurs, then at least one vertex located in \(B(r)\) belongs to a component of Euclidean diameter no less than \(r\). The result becomes strongest, when neither \(\zeta\) nor \(\xi\) depend on \(\lambda\) and the respective properties hold through the whole subcritical annulus-crossing phase.
	
	\paragraph{Relation to subcritical percolation.} 
	Let us relate our result to the \emph{classical percolation phase transition} defined by	
	\[
		\lc=\inf \big\{\lambda>0\colon \P_{\lambda}\big(\exists \text{ an infinite connected component in }\scrG)>0\big\}.
	\]
	Again, \(\lc\) is only guaranteed to be the unique such phase transition if the graph is monotone with respect to $\lambda$. Clearly, any existing infinite connected component must intersect infinitely many annuli, due to stationarity and local finiteness; hence \(\wlc\leq \lc\). 

    \begin{corollary}[Existence of a subcritical percolation phase]
		If the uniform mixing Property~\ref{G:mixing}\(_\lambda\) and the no-long-edges Property~\ref{G:noLongEdgeUnif}\(_\lambda\) hold for some \(\lambda>0\), then \(\lc>0\). 
	\end{corollary}
    We shall see in our examples that the reverse does not generally hold true. More precisely, we present examples where \(\wlc=0\) but still \(\lc>0\). An interesting open question remains as to whether $0<\wlc<\lc$ is possible for some of the models discussed in this work.

 In the following section, we apply our main results to a special but still large class of spatial random graphs and derive adjusted conditions for the existence of subcritical phases.


	\section{The weight-dependent random connection model} \label{sec:wdrcm}
	
	The \emph{weight-dependent random connection model} (WDRCM) is a continuum percolation model that was first introduced in~\cite{GHMM2022} as a generalisation of the standard \emph{random connection model}~\cite{Penrose1991}. In the latter, the connection probability is a decreasing function of the distance between the vertices that are to be connected. In the weight-dependent version, however, vertices additionally carry i.i.d.\ vertex marks that influence the connection probability, thus modelling the vertices' attractiveness. This allows for greater inhomogeneity in the model and notably includes models with power-law degree distribution. By specifying how marks enter the connection probability, one can regain versions of many established models from the literature and we present an extensive list below in Table~\ref{tab:interPol}. 
	
	To introduce the model formally, the WDRCM is the independent setting of Definition~\ref{def:IndS} where the vertex set \(\scrV\) is a standard Poisson point process on \(\R^d\times(0,1)\); each Poisson point is located in \(\R^d\) and is independently marked uniformly on \((0,1)\). Two given vertices \(\x=(x,u_x)\) and \(\y=(y,u_y)\) are then connected with probability
	\begin{equation} \label{eq:classicPhi}
		\mathbf{p}(\x,\y) = \varphi(u_x,u_y,|x-y|^d),
	\end{equation}
	where we assume that \(\varphi:(0,1)\times(0,1)\times(0,\infty)\to[0,1]\) is
	\begin{enumerate}[(i)]
		\item symmetric and non-increasing in the first two arguments and non-increasing in the third argument and 	
		\item integrable, i.e.,
			\[
				\int_0^1 \d u \int_0^1 \d v \int_0^\infty \d r \, \varphi(u,v,r)<\infty.
			\] 
	\end{enumerate}
    Note that the basic properties \ref{G:Point_process} and \ref{G:Translation} are evident, while (ii) implies Property~\ref{G:locallyFinite}. Property~(i) entails that vertices are more likely to be connected by an edge if they are located close to each other or if they have small marks. Hence, smaller marks lead to higher degrees on average and the marks can be thought of as the \emph{inverse weights} of the vertices, giving the model its name. The analogy of a small mark being equivalent to a strong influence is clearest in the special case of the \emph{age-dependent random connection model}, for which this parametrisation was first applied~\cite{GGLM2019}. Here, marks represent the vertices' birth times, and an early birth time indicates an old, thus influential, vertex. Finally, recall that the independent setting implies monotonicity and mixing, and hence an understanding of the occurrence of long edges remains the missing ingredient to apply our main results.
		
	 \subsection{Verifying the no-long-edges condition using the effective decay exponent.} \label{sec:deff}
	 In order to apply our results to WDRCMs, we need to verify the Properties~\ref{G:noLongEdge}\(_\lambda\) and~\ref{G:noLongEdgeUnif}\(_\lambda\). In this context, we can provide relatively simple sufficient conditions for verifying or excluding these properties, based on natural exponents associated with the model. We begin with a heuristic explanation. Consider the two sets \(\mathscr{B}_1\) and \(\mathscr{B}_2\) of vertices, where \(\mathscr{B}_1\) consists of all vertices located within \(B(r)\) and \(\mathscr{B}_2\) consists of those located in the annulus \(B(3r)\setminus B(2r)\). We compute the expected number of edges between these sets using Campbell's formula as
 	\[
 		\int_{B(r)}\d x \int_{B(3r)\setminus B(2r)} \d y \int_0^1 \d u \int_0^1 \d v \, \varphi(u,v,|x-y|^d). 	
 	\]
 	Note that all vertices in \(\mathscr{B}_1\) and \(\mathscr{B}_2\) are roughly distance \(r\) apart, so we may approximate $|x-y|$ by $r$, resulting in the order
 	\[
   		r^{2d}\int_0^1 \d u \int_0^1 \d v \, \varphi(u,v,r^d).
 	\]
 	A first naive approach would be to assume a Poisson structure of the edges and, as a result, estimating the probability of finding an edge between the two regions as \[1-\exp\Big(-r^{2d}\int_0^1 \d u \int_0^1 \d v \, \varphi(u,v,r^d)\Big).\] However, an important contribution to this expectation may come from the strongest vertices ---those with lowest marks--- which are rare but, when present, may connect to many long edges. Put differently, the full expectation may not capture the likelihood of long edges adequately. We refer the reader to~\cite{GraLuMo2022} for a more detailed discussion on this matter. 
 	
 	To address this, we introduce a free parameter $\mu\in (-\infty,1)$ and instead count:
	\begin{enumerate}[(i)]
    	\item the number of \emph{edges} between `weak' vertices with marks at least $r^{-d+d\mu}$, and 
    	\item the number of `strong' \emph{vertices} with marks at most $r^{-d+d\mu}$.
	\end{enumerate}
	On average, there are of order
	\[
		I(\mu,r):=r^{2d}\int_{r^{-d+d\mu}}^1 \d u \int_{r^{-d+d\mu}}^1 \d v \, \varphi(u,v,r^d)
	\]
	such edges between weak vertices, and of order $r^{d\mu}$ such strong vertices. We then select a $\mu$ that minimises the combined contributions. More precisely, we define
 	\begin{equation}\label{eq:psi}
 		\psi(\mu) := \limsup_{r\to\infty} \frac{\log \Big(\int_{r^{-d+d\mu}}^1 \d u \int_{r^{-d+d\mu}}^1 \d v \  \varphi(u, v, r^d)\Big)}{d\log r}
 	\end{equation}
 	and
 	\begin{equation} \label{eq:zeta}
 		\zeta :=\sup\{\mu\in (-\infty,1)\colon \mu < 2+\psi(\mu)\}.
 	\end{equation}
 	It is worth noting here that \(\psi\) is a non-increasing function with values in \([-\infty,-1]\). 
  Thus, $\zeta$ is well-defined, and $I(\zeta, r)+r^{d\zeta}$ yields the smallest order in $r$. Furthermore, for any $\mu>\zeta$, we have \(r^d/(r^{-d+d\mu}-r^{-d+d\zeta})=r^{2d-d\mu}/(1-r^{d(\zeta-\mu)})\), and therefore by monotonicity of \(\varphi\) in the mark arguments
     \begin{equation}\label{eq:expectedEdges}
     	\begin{aligned}
     		r^{d}\hspace{-0.18 cm} \int\limits^1_{r^{-d+d\mu}}\hspace{-0.18 cm}\d v \ \varphi(r^{-d+d\mu},v,r^d)
     		&
     			\leq \tfrac{r^{2d-d\mu}}{1-r^{d(\zeta-\mu)}} \hspace{-0.18 cm}\int\limits_{r^{-d+d\mu}}^1 \hspace{-0.18 cm} \d v \hspace{-0.18 cm}\int\limits_{r^{-d+d\zeta}}^{r^{-d+d\mu}} \hspace{-0.18 cm} \d u \ \varphi(u,v,r^d)
     			\leq \tfrac{1}{1-r^{d(\zeta-\mu)}} r^{-d\mu}I(\mu,r) 
     			<1,
     	\end{aligned}
     \end{equation}
     for large \(r\). This indicates that a vertex in \(\mathscr{B}_1\) with mark $u=r^{-d+d\mu}$ will, on average, only connect by a long edge to fewer than one weaker vertex in \(\mathscr{B}_2\) with mark at least \(u\). Conversely, a vertex with mark $u=r^{-d+d\mu}$ with $\mu<\zeta$ will typically be incident to more than one long edge with weaker target vertex. This characteristic naturally motivates the definition of $\zeta$ as it essentially estimates the expected number of vertices in \(\mathscr{B}_1\) incident to at least one long edge with weaker target vertex. This very definition is actually used in~\cite{jorritsmaKomjathyMitsche2024LDP,JorritsmaKomjathyMitsche2023} to introduce the parameter $\zeta_{\text{long}}$, which quantifies the cluster-size decay of finite components in supercritical WDRCMs in the regimes where long edges appear frequently (\(\zeta_\text{long}>0\)). It can be verified that we always have $\zeta_{\text{long}}=\zeta$. However, our alternative definition of $\zeta$ is more transparent for our approach and heuristics, particularly facilitating the generalisations to more correlated models in Section~\ref{sec:genWDRCM} below.
     
     We further elaborate on the heuristic link between $\zeta$ and $\P_\lambda(\cL(r,1))$. Suppose first that $\zeta<0$. In this case, the expectation above is approximately $r^{d\zeta}\ll 1$, implying that the probability of finding a long edge is bounded by $r^{d\zeta}$, yielding Property~\ref{G:noLongEdge}\(_\lambda\).

	Suppose now that $\zeta>0$. Then, the expectation above is approximately $r^{d\zeta}\gg 1$, so we typically expect to find many long edges. Assuming sufficient continuity to have $\zeta=2+\psi(\zeta)$ for simplicity, then long edges should only be absent if either
	\begin{enumerate}[(i)]
		\item no vertex has mark smaller than \(r^{-d+d\zeta}\), which occurs with probability roughly $\exp(-r^{d\zeta})$, or
		\item there is no edge between `weak' vertices with larger marks.
	\end{enumerate} 
	If we accept that $r^d \1_{[r^{-d+d\zeta},1]}$ is well enough approximated by the empirical distributions of these weaker vertices marks, then the number of such edges should be given by a Poisson random variable with mean $I(\zeta,r)$. Thus, the probability of finding no such edge is roughly $\exp(-r^{d(2 +\psi(\zeta))})=\exp(-r^{d\zeta})$. Formalising these heuristics, we obtain the following theorem.
	
	\pagebreak

 \begin{theorem}[Subcritical annulus-crossing phases for WDRCMs] \label{thm:WDRCM}
 	Let \(\scrG\) be the weight-dependent random connection model.
 	\begin{enumerate}[(i)]
 		\item 
 			If \(\zeta<0\), then
 			\[
 				\limsup_{r\to\infty}\frac{\log\P_1(\cL(r,1))}{\log r} = {d\zeta}
 			\]
 			and thus \(\wlc>0\). Moreover, for all \(\lambda<\wlc\), we have
 			\[
 				\limsup_{r\to\infty}\frac{\log \P_\lambda\big(B(r)\xleftrightarrow[]{}B(2r)^{\textsf{c}}\big)}{\log r} = d\zeta.
 			\]
 		\item 
 			If \(\zeta>0\), then 
 			\[
 				\limsup_{r\to\infty}\frac{\log\  \lvert\log \P_{1}(\neg \cL(r,1))\rvert}{\log r}\ge d\zeta>0.
 			\] 
 			In particular, \ref{G:noLongEdge}\(_\lambda\) does not hold for any \(\lambda>0\) and thus \(\wlc=0\).
 	\end{enumerate}
 \end{theorem}
 
 We give the proof of Theorem~\ref{thm:WDRCM} in Section~\ref{sec:proofsExamples}.
 
 \begin{remark}\label{rem:thmWDRCM}
 	Let us comment on the above result. 
 	\begin{enumerate}[(a)]
 		\item 
 			An equivalent statement to Part~(ii) is that, for all \(\varepsilon>0\), the event \(\cL(r,1)\) occurs with probability at least $1-\exp(-r^{d\zeta-\varepsilon})$ for \emph{some}  arbitrarily large values of \(r\). Additionally, if the \(\limsup\) defining $\psi$ is a true limit, then  the \(\limsup\) in Part~(ii) can also be replaced by a \(\liminf\), that is, \(\cL(r,1)\) occurs with probability at least $1-\exp(-r^{d\zeta-\varepsilon})$ for \emph{all} large values of $r$. Similarly, under this condition, the \(\limsup\) in the statements of Part~(i) are true limits, or in other words, the probabilities are of order $r^{d\zeta + o(1)}.$
		\item 
  			To apply our results effectively, we assume \(\zeta\neq 0\). The case where \(\zeta=0\) is more intricate and may be described as \emph{weakly scale invariant}, akin to the scale-invariant regime in long-range percolation. In this regime, results can depend heavily on specific details of the model, so universal results may not hold. Consider, for example, scale-invariant long-range percolation, where each pair of vertices at distance \(r\) is independently connected by an edge with probability \(r^{-2d}\). It is straightforward to see that \(\zeta=0\) (also see Proposition~\ref{prop:deffForWDRCM} below). By independence of edges, the number of long edges connecting the sets \(\mathscr B_1\) and \(\mathscr B_2\) (as defined above) is binomially distributed with parameters of order \(r^{2d}\) and \(r^{-2d}\). Hence, there exists a long edge with positive probability, implying \(\wlc=0\). However, if we replace the connection probability by \((r \log(r))^{-2d}\), then still \(\zeta=0\), but the expected number of long edges tends to zero, implying \(\wlc>0\).
  			While for many models \(\zeta=0\) aligns with boundary regimes, the \emph{age-dependent random connection model}~\cite{GGLM2019} is an important example that has \(\zeta=0\) for a whole non-empty parameter regime, independently of the profile-function, cf.~Figure~\ref{fig:Deff} below. For a more detailed discussion on the weakly scale-invariant case, we refer the reader to~\cite{GraLuMo2022}.			
 	\end{enumerate}
 \end{remark}

 The theorem highlights the importance of determining the value of $\zeta$ and, in particular, whether $\zeta$ is positive or negative. Let us define
 \[
    \deff^+:=-\psi(0+)\qquad\text{ and }\qquad  \deff^-:=-\psi(0-),
 \]
 the limit from the right and left respectively of the absolut value of $\psi$ at zero. Then $\deff^->2$ implies $\zeta<0$, while $\deff^+<2$ implies $\zeta>0$. If $\psi$ is continuous at zero, we can simply define
 \[
 \deff:=\deff^+=\deff^-=-\psi(0)=- \limsup_{r\to\infty} \frac{\log \Big(\int_{r^{-d}}^1 \d u \int^1_{r^{-d}} \d v \  \varphi(u, v, r^d)\Big)}{d\log r}.
 \]
 The parameter $\deff$ was first introduced in~\cite{GraLuMo2022} as the \emph{effective decay exponent} in order to derive the existence or absence of supercritical percolation phases in one-dimensional WDRCMs. It was further used in~\cite{Moench2024} to prove continuity of the percolation function in the \(\deff<2\) regime. The motivation for introducing $\deff$ comes from the observation that the smallest vertex marks among the vertices in \(\mathscr B_1\) and \(\mathscr B_2\) (defined as above as ball and annulus, respectively) are typically of order $r^{-d}$, suggesting that the probability of two randomly picked vertices in \(\mathscr B_1\) and \(\mathscr B_2\) being connected by an edge scales like \(r^{-d\deff}\). 
 
 This leads to a natural analogy with a classical long-range percolation model, where each pair of vertices \(\x,\y\) is independently connected by an edge with probability proportional to \(|x-y|^{-d\deff}\), which justifies the term \emph{effective decay exponent} for \(\deff\). Our results demonstrates that this simple exponent $\deff$ largely dictates the principal behaviour of the model and distinguishes the two universality classes, analogous to classical long-range percolation~\cite{Schulman1983,Penrose1991}. However, while \(\deff\) captures the model's overall behaviour, it may lack precision for finer results such as explicit tail bounds. In particular, the specific manner in which vertex marks contribute can be too nuanced to be fully characterised by \(\deff\) alone leading us to the finer exponent \(\zeta\). 
 
 It is worth noting that the two universality classes according to strong (\(\zeta>0\)) or weak (\(\zeta<0\)) long-range effects are insufficient to fully describe all global properties accurately. Indeed, while \(\zeta\) effectively captures the overall occurrence of long edges, it does not capture their precise connectivity. Specifically, in some models the \(\zeta>0\) phase includes parameter regimes with \(\lc>0\) as well as those with \(\lc=0\)~\cite{GLM2021}. Detecting such a change of behaviour requires a deeper understanding and control over the inhomogeneities arising from the marks.
 
Finally, let us note that although we have exclusively assumed vertex locations as governed by a Poisson point process throughout the section, all results remain applicable to a Bernoulli site-percolated lattice. In this case, the integrals over vertex locations should be replaced by sums. 

\subsubsection{The interpolation model} \label{sec:wdrcmInterpol}
  \begin{table}
\begin{center}
	\caption{Various choices for \(\gamma\), \(\alpha\) and \(\delta\) for the weight-dependent random connection model and the models they represent in the literature together with their \(\zeta<0\) phases and the value of \(\zeta\) within. Here, to shorten notation, \(\delta=\infty\) represents models constructed with \(\rho\) being the indicator function.}
	\small
	\begin{tabular}{l l l l}
		\toprule
	 	\textbf{Parameters} & \(\boldsymbol{\zeta<0}\) \textbf{iff} & \(\boldsymbol{\zeta}=\) & \textbf{Names and references} 
	 	\tabularnewline 
	 	\midrule  
	 	\(\gamma = 0, \alpha = 0, \delta = \infty\) & always & \(-\infty\) & Gilbert graph~\cite{Gilbert61}, \\ & & & random geometric graph~\cite{Penrose2003} \vspace{4 pt}
	 	\tabularnewline 
	 	\(\gamma = 0, \alpha =0, \delta<\infty\) & \(\delta>2\) & \(2-\delta\) & random connection model~\cite{Penrose1991}, \\ & & & long-range percolation~\cite{Schulman1983} \vspace{4 pt}
	 	\tabularnewline 
	 	\(\gamma>0, \alpha=0,\delta=\infty\) & \(\gamma<1\) & \(1-\tfrac{1}{\gamma}\) & Boolean model \cite{Hall85,Gouere08}, \\ & & & scale-free Gilbert graph~\cite{Hirsch2017} \vspace{4 pt}
	 	\tabularnewline 
	 	\(\gamma>0, \alpha = 0, \delta<\infty\) & \(\delta>2\), \(\gamma<\tfrac{\delta-1}{\delta}\) & \((1-\tfrac{\delta-1}{\gamma\delta})\vee (2-\delta)\) & soft Boolean model~\cite{GGM22,jahnel_Lu_Ort_2024_cluster} \vspace{4 pt}
	 	\tabularnewline 
	 	\(\gamma = 0, \alpha>0, \delta = \infty\) & \(\alpha<1\) & \(-\infty\) & ultra-small scale-free geometric  \\ & & & network (\(\alpha>1\))~\cite{Yukich2006}, \\ & & & weak-kernel model~\cite{GHMM2022,Lue2022}  \vspace{4 pt}
	 	\tabularnewline 
	 	\(\gamma>0, \alpha=\gamma, \delta\leq \infty\) & \(\delta>2\), \(\gamma<\tfrac{1}{2}\) & \(\tfrac{\delta(2\gamma-1)}{2\delta\gamma-1}\vee (2-\delta)\); \(\delta<\infty\) &scale-free percolation \cite{DeijfenHofstadHooghiemstra2013,DeprezWuthrich2019}, \\ & & \(\tfrac{2\gamma-1}{\gamma}\); \(\delta=\infty\) & geometric inhomogeneous \\ & & & random graphs ~\cite{BringmannKeuschLengler2019} \vspace{4 pt}
	 	\tabularnewline 
	 	\(\gamma>0, \alpha=1-\gamma, \delta\leq\infty\) & never & not applicable  & age-dependent random \\ & & & connection model~\cite{GGLM2019}
	 	\tabularnewline \bottomrule 
	\end{tabular}
	\label{tab:interPol}
\end{center}
\end{table} 
	We apply the previously established results and, in particular, derive values of \(\zeta\) for a specific yet general instance of the WDRCM. More precisely, we discuss the \emph{interpolation model} first introduced in~\cite{GraLuMo2022} as a model that contains and interpolates between many important and well-established models from the literature. Recall that \(\scrV\) is given by a standard Poisson point process on \(\R^d\times(0,1)\) and that we denote vertices by \(\x=(x,u_x)\). In the interpolation model, edges are drawn independently given \(\scrV\) with probability
 \begin{equation}\label{eq:rho}
 	\mathbf{p}(\x,\y)=\varphi(u_x,u_y,|x-y|^d)=\rho\big(g(u_x,u_y)|x-y|^d\big),
 \end{equation}
 where \(\rho\colon(0,\infty)\to[0,1]\) is an integrable and non-increasing \emph{profile function} and 
 \begin{equation}\label{eq:interpolKernel}
 	 g\colon (0,1)\times (0,1)\to (1,\infty), \quad (s,t) \mapsto  (s\wedge t)^{\gamma}(s\vee t)^{\alpha},
 \end{equation}
 for \(\gamma\in[0,1), \alpha\in[0,2-\gamma)\), is the \emph{interpolation kernel}. Note that the interpolation kernel \(g\) is non-increasing and symmetric in both arguments. It is straightforward to deduce that the interpolation model is a WDRCM with all required properties. Further, regardless of the profile-function, the degree distribution of the interpolation model follows a power law if either \(\gamma>0\) or \(\alpha>1\), see~\cite{Lue2022}. 
    
Let us now examine the profile function more closely. Two types have been established in the literature: The \emph{long-range} profile function \(\rho(x):=p(1\wedge |x|^{-\delta})\) for \(\delta> 1\) or the \emph{short-range} profile function \(\rho(x):=p\mathbbm{1}\{0\le x\le 1\}\). Here, \(p\in(0,1]\) can be used to simultaneously introduce Bernoulli bond-percolation. However, we will mainly assume \(p=1\) as it does not qualitatively affect our results. As summarised in~Table~\ref{tab:interPol}, the interpolation model can be used to describe versions of many well-established models. This model has the huge advantage that it describes these models with only four real parameters and allows for easy comparisons and couplings. A similar, yet slightly different, parameterisation has also been studied, for instance, in~\cite{HofstadHoornMaitra2023, jorritsmaKomjathyMitsche2024LDP,JorritsmaKomjathyMitsche2023}; in Table~\ref{tab:translation} we present a translation between the two parameterisations.  
 \begin{table}[ht!]
	\begin{center}
		\caption{Translation of parameters in kernel-based spatial random graphs and the classical WDRCM.}
		\begin{tabular}{c c c}
			\toprule
			parameters in~\cite{JorritsmaKomjathyMitsche2023,jorritsmaKomjathyMitsche2024LDP,HofstadHoornMaitra2023} & & parameters here and in~\cite{GHMM2022,GLM2021,Moench2024}
			\tabularnewline \midrule
			\(\tau - 1\) & \phantom{\(\hat{=}\)} & \(1/\gamma\) 
			\tabularnewline 
			\(\alpha\) & \phantom{\(\hat{=}\)} & \(\delta\) 
			\tabularnewline
			\(\sigma\) & \phantom{\(\hat{=}\)} & \(\alpha/\gamma\)
			\tabularnewline
			\bottomrule 
		\end{tabular}
		\label{tab:translation}
	\end{center}
\end{table} 

The following proposition summarises the values \(\zeta\) for the interpolation model with a long-range profile function. However, the according results for a short-range profile can simply be derived by sending \(\delta\to\infty\). Its proof is given in Section~\ref{sec:proofsExamples}.

\begin{prop}\label{prop:deffForWDRCM} 
	Consider the interpolation model with a profile function with \(\delta\in(1,\infty)\).
    \begin{enumerate}[(i)]
        \item 
            If \(\delta>2\), \(\gamma<1-1/\delta\), and \(\alpha<1-\gamma\), then
            \[
                \zeta = \max\big\{2-\delta, 1-\tfrac{\delta-1}{\gamma\delta}, \tfrac{\alpha+\gamma-1}{\gamma},\tfrac{2(\alpha+\gamma-1)}{\alpha+\gamma}\big\} <0.
            \]
        \item  
            If \(\delta=2\), \(\gamma<1-1/2\), and \(\alpha<1-\gamma\), or if \(\delta\geq 2\) and either \(\gamma<1-1/\delta\) and \(\alpha=1-\gamma\), or \(\gamma=1-1/\delta\) and \(\alpha<1-\gamma\), then \(\zeta=0\).
        \item   
            If either \(\delta<2\), or \(\gamma>1-1/\delta\), or \(\alpha>1-\gamma\), then
            \[
                \zeta= \max\big\{2-\delta, 1-\tfrac{\delta-1}{\gamma\delta}, \tfrac{\delta(\alpha+\gamma-1)}{\delta(\alpha+\gamma)-1}\big\}>0.
            \]
    \end{enumerate}
\end{prop}

\begin{figure}
\centering
\begin{subfigure}{0.45\textwidth}
    \resizebox{\textwidth}{!}{
    \begin{tikzpicture}[every node/.style={scale=1.2}]
        \draw[->] (0,0) to (13,0) node[right] {$\gamma$};
        \draw	
    		  (10,0) node[anchor=north] {1};

        \draw[dotted] (-2.7,10) to (10,10);
        \draw[] (10,0) to (10,10)
	           (10,10) to (0,13);

        \draw[](0,0) to (0,10.5);
        \draw [->] (0,12.5) to (0,13.3) node[above] {$\alpha$};
        \draw[decorate, decoration = {snake, segment length = 10 pt, amplitude = 1mm}] (0, 10.5)--(0,12.5);
        \draw	(0,0) node[anchor=east] {0}
        	   (0,10) node[anchor=east] {1}
        	   (0,13) node[anchor = east] {2};

        \draw (6,-0.7) node[align = left, anchor = north] {Boolean model};
        \draw (-1.7,11.5) node[align = left] {ultra-small \\ scale-free \\ geometric \\ network};
        \draw (12,2.3) node[align = left] {age-dependent \\ random connection \\ model};
        \draw[->, bend angle = 45, bend right] (11.5,3) to (6,4.4);
        \draw (12,7.5) node[align = left] {scale-free \\ percolation};
        \draw[->, bend angle = 45, bend left] (11.5,7) to (7.2,7);
        \draw (0,0) node[circle, fill = black, scale=0.4] {};
        \draw (-0.4,-0.7) node[align = left, anchor = north] {random connection model};
        \draw[->] (-0.5, -0.5) to (-0.2,-0.2);

        \draw 	(7, 0) node[anchor = north] {$\tfrac{\delta-1}{\delta}$}
			 (3, 0) node[anchor = north] {$\tfrac{1}{\delta}$} 
			 (0, 3) node[anchor = east] {$\tfrac{1}{\delta}$}
			 (0, 6) node[anchor = east] {$\tfrac{2}{\delta}$};

        \draw[pattern=north east lines, pattern color=lightgray!30!, draw=none] 
            (0,0) plot[smooth,samples=200,domain=0:7,variable=\g]({0},{10*\g/7}) -- 
            plot[smooth,samples=200,domain=7:0,variable=\g]({7},{3*\g/7});
         
   	    \draw[thick] (7,0) to (7,3); 
	   \draw[thin] (3,3) to (7,3); 
	   \draw[thin] (7,3) to (10,3); 
	   \draw[thin] (3,0) to (3,3); 
	   \draw[thick] (0,10) to (7,3); 
        \draw[dashed] (7,3) to (10,0); 
        \draw[dashed] (0,0) to (3,3); 
        \draw[thin] (3,3) to (5,5); 
        \draw[dashed] (5,5) to (10,10); 
        \draw[thin] (0, 6) to (3,3); 
 
        \draw	(6.5,8) node[scale = 1.5, thick] {$\tfrac{\delta(\alpha+\gamma-1)}{\delta(\alpha+\gamma)-1}>0$}
    		  (8.5,1.5) node[scale = 1.3, thick, rotate = 0] {$1-\tfrac{\delta-1}{\gamma\delta}>0 $}
    		  (1.5,3) node[scale = 1.3, thick] {$2-\delta<0$}
    		  (5,1.5) node[scale = 1.3, thick, rotate = 0] {$1-\tfrac{\delta-1}{\gamma\delta}<0$}
    		  (2.8,5.7) node[scale = 1.3, thick, rotate = -45] {$\tfrac{2(\alpha+\gamma-1)}{\alpha+\gamma}<0$}
    		  (5,3.5) node[scale = 1, thick, rotate = 0] {$\tfrac{\alpha+\gamma-1}{\gamma}<0$};
    \end{tikzpicture}
    }
    \caption{\(\delta>2\)} 
    \label{fig:DeffGeq2}
\end{subfigure}
\hspace{1cm}
\begin{subfigure}{0.45\textwidth}
    \resizebox{\textwidth}{!}{
    \begin{tikzpicture}[every node/.style={scale=1.2}]
        \draw[->] (0,0) to (13,0) node[right] {$\gamma$};
        \draw	
    		    (10,0) node[anchor=north] {1};

        \draw[dotted] (-2.7,10) to (10,10);
        \draw[] (10,0) to (10,10)
	            (10,10) to (0,13);
    
        \draw[](0,0) to (0,10.5);
        \draw [->] (0,12.5) to (0,13.3) node[above] {$\alpha$};
        \draw[decorate, decoration = {snake, segment length = 10 pt, amplitude = 1mm}] (0, 10.5)--(0,12.5);
        \draw	(0,0) node[anchor=east] {0}
        	   (0,10) node[anchor=east] {1}
        	   (0,13) node[anchor = east] {2};

        \draw (6,-0.7) node[align = left, anchor = north] {Boolean model};
        \draw (-1.7,11.5) node[align = left] {ultra-small \\ scale-free \\ geometric \\ network};
        \draw (12,2.3) node[align = left] {age-dependent \\ random connection \\ model};
        \draw[->, bend angle = 45, bend right] (11.5,3) to (6,4.4);
        \draw (12,7.5) node[align = left] {scale-free \\ percolation};
        \draw[->, bend angle = 45, bend left] (11.5,7) to (7.2,7);
        \draw (0,0) node[circle, fill = black, scale=0.4] {};
        \draw (-0.4,-0.7) node[align = left, anchor = north] {random connection model};
        \draw[->] (-0.5, -0.5) to (-0.2,-0.2);

        \draw 	
			    (5.55, 0) node[anchor = north] {$\tfrac{1}{\delta}$} 
			    (0, 5.55) node[anchor = east] {$\tfrac{1}{\delta}$}
			    (0, 11.1) node[anchor = east] {$\tfrac{2}{\delta}$};

         
	   \draw[dashed] (0,10) to (7,3); 
        \draw[dashed] (7,3) to (10,0); 
        \draw[dashed] (0,0) to (3,3); 
        \draw[dashed] (3,3) to (5,5); 
        \draw[dashed] (5,5) to (10,10); 
        \draw[thin] (5.55,0) to (5.55, 5.55);
        \draw[thin] (10,5.55) to (5.55,5.55);
        \draw[thin] (0, 11.1) to (5.55,5.55);

        \draw	(6,8) node[scale = 1.5, thick] {$\tfrac{\delta(\alpha+\gamma-1)}{\delta(\alpha+\gamma)-1}>0$}
    		  (8,3) node[scale = 1.5, thick, rotate = 0] {$1-\tfrac{\delta-1}{\gamma\delta}>0 $}
    		  (2.5,5) node[scale = 1.5, thick] {$2-\delta>0$};
    \end{tikzpicture}
    }
    \caption{\(\delta\in(1,2)\)} 
    \label{fig:DeffLeq2}
\end{subfigure}
\caption{Phase diagram in \(\gamma\) and \(\alpha\) for the interpolation model constructed with a long-range profile with \(\delta>2\) in~(a) and \(\delta\in(1,2)\) in~(b). The \(\zeta<0\) phase in~(a) is shaded in grey while the \(\zeta>0\) phases are not shaded. The values of \(\zeta\) in the corresponding parameter regimes are shown. The solid line in~(a) marks the phase transition \(\zeta=0\). Dashed lines represent no change of behaviour.}
\label{fig:Deff}
\end{figure}
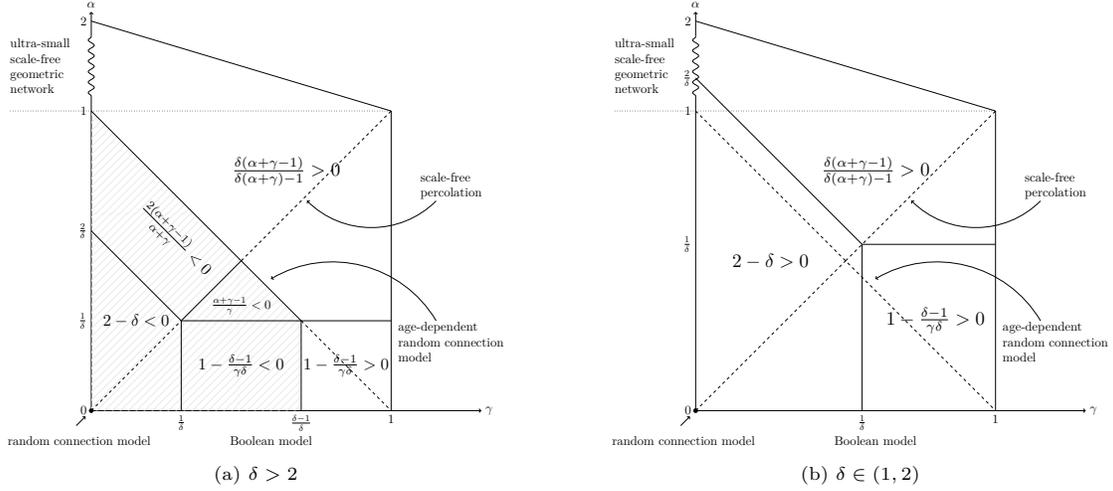

	Our findings are summarised in Figure~\ref{fig:Deff}, and we further apply Part~(i) to various models from the literature in Table~\ref{tab:interPol}.

    Let us further discuss the possible subcritical annulus-crossing phase, depending on the value of $\zeta$. First, when $\zeta>0$, by Theorem~\ref{thm:WDRCM} the long-edge probability \(\P_1(\cL(r,1))\) does not decay to zero but actually approaches one. Observing that the WDRCM is a model in the independent setting, by Theorem~\ref{thm:equiStatement} this is equivalent to saying that \(\wlc=0\), and there is no subcritical annulus-crossing phase.

     When $\zeta=0$, as outlined in Remark~\ref{rem:thmWDRCM}~(b), whether \(\wlc=0\) or \(\wlc>0\) depends on fine details of the model. In fact, an additional slowly varying correction term may influence that substantially. However, in the strict Pareto case of our parameterisation of the interpolation model, it is easily seen that we still have \(\wlc=0\), as was first observed in~\cite{jacob2025crossing} for the special instance of scale-free percolation. 
	 \begin{prop}\label{prop:zeta=0}
	 	Consider the interpolation model with \(\delta=2\) or with \(\delta>2\) and either \(\gamma=1-1/\delta\) and \(\alpha \leq 1-\gamma\) or \(\gamma<1-1/\delta\) and \(\alpha=1-\gamma\), then \(\P_1(\cL(r,1))>0\) and consequently \(\wlc=0\).
	 \end{prop}

We present the proof in Section~\ref{sec:proofsExamples}.
	Let us finally comment on the more interesting \(\zeta<0\) regime of the interpolation model. We can again use Theorem~\ref{thm:WDRCM} and Theorem~\ref{thm:equiStatement} to deduce that we have a non-trivial subcritical annulus-crossing phase, namely \(\wlc>0\), in that case. Moreover, it is easy to see that the \(\limsup\) defining $\psi$ is a true limit, so \(\P_1(\cL(r,1))= r^{d\zeta +o(1)}\to 0\) by Remark~\ref{rem:thmWDRCM}~(a). In the whole subcritical annulus-crossing phase \(\lambda\in(0,\wlc)\), the annulus-crossing probabilities also decay at the same speed, and one might thus expect that the subcritical behaviour is driven by \(\P_\lambda(\cL(r,1))\) and explained by \(\zeta\) alone. 

    As already mentioned below Theorem~\ref{thm:Diamter}, the probability of the annulus-crossing event also determines the largest Euclidean diameter of the clusters in \(B(r)\). In the classical WDRCM this further applies to the tail of the Euclidean diameter of the typical cluster of the origin when the origin is added with a uniformly distributed vertex mark to the graph. Hence,
		\[
			\P^o_\lambda(\exists \x\colon |x|>r, \o\xleftrightarrow[]{}\x) \leq \P_\lambda \big(B(r)\xleftrightarrow[]{} B(2r)^\textsf{c}\big)  = r^{d\zeta+o(1)}.
		\]  
		Furthermore, the tail of the Euclidean diameter can provide bounds on the cardinality of the cluster of the origin as the Poisson point process is well-concentrated. This relation becomes particularly strong for the standard Boolean model (\(\alpha=0, \delta\to \infty\)), for which it is shown in~\cite{Gouere08} that for all \(\lambda<\wlc\),   
			\[
				\P_\lambda\big(\sharp\{\x: \o\xleftrightarrow[]{}\x\}>r^d\big)\asymp \P_\lambda\big(\exists \x: |x|>r, \o\xleftrightarrow[]{}\x\big) \asymp \P_\lambda(\cL(r,1)). 
			\]
	 	Further, it was recently shown that \(\lc=\wlc\) in the standard Boolean model for all but at most countably many values of \(\gamma\)~\cite{DembinTassion2022}. The potential exclusion of countably many values of \(\gamma\) results from the techniques applied in the proof, though it is widely believed that both phase transitions are always the same. These findings combined appear to strengthen the above idea and suggest that even the whole subcritical behaviour of the Boolean model can be described solely by the tail of the long-edges event. However, this is not the case in general. For instance, consider long-range percolation (\(\gamma=\alpha=0\) and \(\delta<\infty\)). Thus, \(\wlc = 0\) whenever \(\delta<2\) according to Proposition~\ref{prop:deffForWDRCM}. However, it is well-known that \(\lc>0\) for all \(\delta>1\), see~\cite{MeesterRoy1996}, implying \(0=\wlc <\lc\). For \(\delta>2\) and \(\lambda<\wlc\) our results yield the polynomial decay with exponent \(2-\delta\) as a bound for the Euclidean diameter and the cardinality of the origin's component. However, the actual tail exponent of the Euclidean diameter is \(1-\delta\) due to~\cite{jahnel_Lu_Ort_2024_cluster} while the cardinality is light-tailed~\cite{MeesterRoy1996}. Thus, describing the subcritical phase solely via annulus crossings is no longer accurate due to the long-range effects inherent in the model. The same can further be observed in the \emph{soft Boolean model}, a model that interpolates between the classical Boolean model and long-range percolation. For a detailed discussion on this model, we refer the reader to~\cite{jahnel_Lu_Ort_2024_cluster}. This of course does not give insights into the already mentioned question as to whether \(0<\wlc<\lc\) is possible. It is also of interest to determine whether the subcritical tail distribution of diameter and cardinality as mentioned above is generally governed by \(\P_1(\cL(r,1))\) for any kernel combined with a short-range profile or whether this is specific to the Boolean model, as well as the correct type of connection for an accurate description of the subcritical phase if this is not the case.

\subsection{Generalisations} \label{sec:genWDRCM}
The previous results of this section and particularly Theorem~\ref{thm:WDRCM} apply our main results to the classical WDRCM. The results get particularly strong as the classical model takes place in the independent setting of Definition~\ref{def:IndS}. Our main results however allow for generalisations towards more correlated models. There are two natural ways to implement correlations into the framework of the WDRCM: either by replacing the underlying Poisson point process with a correlated point process or by extending the connection mechanism in a way that the presence of an edge additionally depends on vertices surrounding its end vertices. Here, we focus on the latter approach and stick with the Poisson process for vertex placement, but we also provide examples of correlated vertex sets in the Sections~\ref{sec:Cox} and~\ref{sec:FRI}. Both effects could, of course, be combined, though we have chosen not to, focusing on presenting the essential concepts in a concise way. 

The setting of the generalised WDRCM is as follows: The vertex set \(\scrV\) remains a standard Poisson point process on \(\R^d\times(0,1)\). Given \(\scrV\), a pair \(\x=(x,u_x),\y=(y,u_y)\in\scrV\) of vertices is connected by an edge with probability 
	\[
		\mathbf{p}(\x,\y,\scrV\setminus\{\x,\y\})=\mathbf{p}(\y,\x,\scrV\setminus\{\x,\y\}),
	\] 
	where we assume that \(\mathbf{p}\) is invariant under translations. We further impose the following homogeneity condition when \(\mathbf{p}\) is integrated with respect to the underlying Poisson point process: For two deterministically given vertices \(\x\) and \(\y\), we have  
	\begin{equation} \label{eq:varphi}
    	\E_\lambda[\mathbf{p}(\x,\y,\scrV)]= \varphi_\lambda(u_x, u_y, |x-y|^d),
	\end{equation}
	where \(\varphi_\lambda\colon (0,1)\times(0,1)\times(0,\infty)\to[0,1]\) is an integrable function that is symmetric in the vertex marks arguments. Note that these deterministically given vertices are no elements of the Poisson point process almost surely. Put differently, the occurrence of an edge between any two vertices does additionally depend on (potentially) the entire vertex set. However, the connection probability between two given vertices \(\x\) and \(\y\) resembles, on average, the connection probability of the same fixed vertices in a classical WDRCM. Nevertheless, the annealed connection function \(\varphi_\lambda\) now depends on \(\lambda\) and we no longer assume monotonicity in the vertex marks arguments. It is thus clear that our basic Assumptions \ref{G:Point_process}--\ref{G:locallyFinite} are always satisfied but one further has to verify one kind of mixing condition and monotonicity (if needed) separately for each connection rule \(\mathbf{p}\). 
	
	In this paragraph, we focus on identifying instances of the models, where long edges are absent. To this end, we demonstrate that the results previously derived can be adapted in a straightforward manner. We define, in the same way as in~\eqref{eq:psi} and~\eqref{eq:zeta},
	\begin{equation*}
		\psi_\lambda(\mu):= \limsup_{r\to\infty}\frac{\log \Big(\int_{r^{-d+d\mu}}^1 \d u \int_{r^{-d+d\mu}}^1 \d v \  \varphi_\lambda(u, v, r^d)\Big)}{d\log r},
	\end{equation*} 
	as well as
	\begin{equation*}
		\zeta_\lambda := \sup\big\{\mu\in(-\infty,1)\colon\mu<2+\psi_\lambda(\mu)\big\},
	\end{equation*}
	noting that both quantities now depend on \(\lambda\). The following results shows that our heuristic explained above remains valid when long edges are rare.
	
	\pagebreak
	
	\begin{theorem} \label{thm:gWDRCM}
		Let \(\scrG\) be a generalised weight-dependent random connection model
		\begin{enumerate}[(i)]
			\item 
				If \(\zeta_\lambda<0\), then 
				\[
					\limsup_{r\to\infty}\frac{\P_\lambda(\cL(r,1))}{\log r}\leq d\zeta_\lambda. 				
				\]
			\item 
				If there exists \(\lambda'>0\) such that \(\overline{\zeta}_{\lambda'}:=\sup\{\zeta_\lambda\colon \lambda<\lambda'\}<0\) and \(\scrG\) has the uniform mixing Property~\ref{G:mixing}\(_{\lambda'}\), then 
				\[
					\limsup_{r\to\infty}\sup_{\lambda<\lambda'}\frac{\P_\lambda(\cL(r,1))}{\log r}\leq d\overline{\zeta}_{\lambda'}.		
				\]
				In particular, we have \(\wlc>0\).
			\item	
				Assume \(\wlc>0\), and for some \(\lambda<\wlc\) the polynomially mixing Property~\ref{G:P_mixing}\(_\lambda^\xi\) for some \(\xi_\lambda\) as well as \(\zeta_\lambda<0\). Then,
				\[
					\limsup_{r\to\infty}\frac{\log \P_\lambda\big(B(r)\xleftrightarrow[]{}B(2r)^\textsf{c}\big)}{\log r}\leq d(\zeta_\lambda \vee \xi_\lambda).
				\]
		\end{enumerate}
	\end{theorem}

	Let us mention that the upper bounds in Theorem~\ref{thm:WDRCM} are special cases of those in Theorem~\ref{thm:gWDRCM}. Specifically, with the previously notation of \(\zeta\) in~\eqref{eq:zeta}, we have \(\overline{\zeta}_\lambda\equiv\zeta\). One of the reasons why we only derive upper but not lower bounds is due to not assuming any monotonicity in the vertex marks arguments of \(\varphi_\lambda\). Nonetheless, the results imply some kind of monotonicity. Indeed, \(\zeta_\lambda<0\) suggest that either vertices with smaller marks are more powerful or the marks play no influential role on the connection probability. 
				
	Let us finally emphasise that  throughout Sections~\ref{sec:wdrcm}, \ref{sec:deff} and~\ref{sec:genWDRCM} we have explicitly used the Euclidean distance in the connection functions and in the derivation of the tail exponent \(\zeta\). This essentially makes the model rotationally invariant. However, the rotation invariance is never used in our proofs and we could easily replace the Euclidean norm by any other norm and adapt our results. In fact, to ensure translation invariance it suffices to assume connection probabilities of the type \(\varphi(u_x,u_y,x-y)\) and our results still remain true as long as there is a way to translate \(x-y\) into some form of distance that allows for meaningful definitions of \(\psi\) and \(\zeta\).


\subsubsection{The soft Boolean model with local interferences} \label{sec:BoolInterferences} 

	We provide an example of a graph that is a generalised WDRCM but no classical WDRCM, which satisfies polynomial mixing and \(\overline{\zeta}_\lambda<0\). The idea is to combine the soft Boolean model~\cite{GGM22,jahnel_Lu_Ort_2024_cluster} with local interference and noise in the spirit of SINR percolation~\cite{DousseEtAl2006, Tobias2020,jahnel2022sinr}. The model has three parameters: \(\gamma\in(0,1), \delta>1\) and \(\beta\in[0,1)\), and the construction is as follows: Each vertex \(\x=(x,u_x)\) has a sphere of influence with radius \(u_x^{-\gamma/d}\). The larger the sphere of influence, the more attractive the vertex is to other vertices. Further, the sphere of influence may be enlarged by additional long-range effects with exponent \(\delta>1\). This setup so far is precisely the soft Boolean model contained in the classical WDRCM discussed above, cf.~Table~\ref{tab:interPol} and~\cite{GGM22,jahnel_Lu_Ort_2024_cluster}. Additionally, each vertex now has a sphere of interference with radius \(u_x^{-\beta/d}\). The vertices within this sphere distract the vertex \(\x\) and complicate its edge creation, cf.~Figure~\ref{fig:softVsInter}. 
 \begin{figure}[t!]
    \centering
    \begin{subfigure}{0.4\textwidth}
        \centering
        \frame{\includegraphics[scale = 0.449]{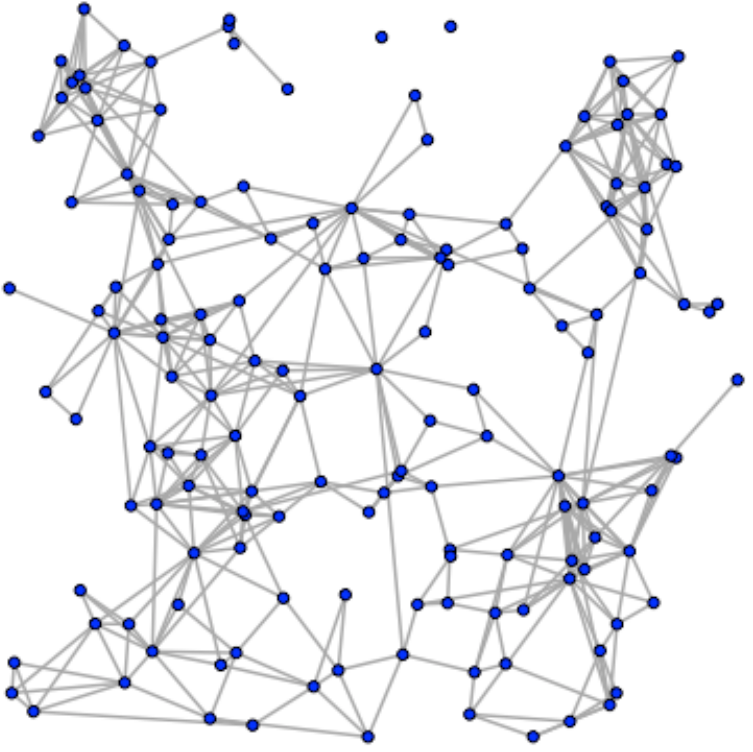}}
        \caption{Soft Boolean model} 
        \label{fig:softVsInterSoft}
    \end{subfigure}
    \hspace{2 cm}
    \begin{subfigure}{0.4\textwidth}
        \centering
        \frame{\includegraphics[scale = 0.449]{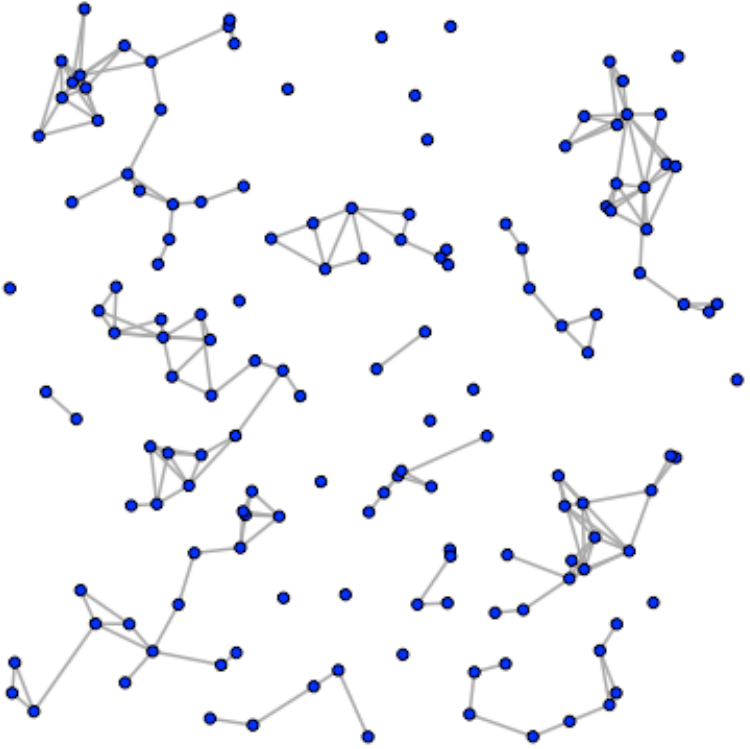}}
        \caption{Soft Boolean model with local interference}
        \label{fig:softVsInterInter}
    \end{subfigure}
    \caption{Examples for the soft Boolean model, Figure~\ref{fig:softVsInterSoft}, and the soft Boolean model with local interference, Figure~\ref{fig:softVsInterInter}, on the same \(150\) vertices sampled from a Poisson point process of intensity \(\lambda=0.04\). For the edge probabilities, the parameters \(\gamma=0.65, \delta=2.7\), and \(\beta=0.3\) are used. Hence, \(\zeta>0\) for the soft Boolean model but \(\zeta<0\) for the model with local interference.}
    \label{fig:softVsInter}
\end{figure}
 More precisely, for a given vertex \(\x=(x,u_x)\), let us introduce the random variable
\[
   \mathscr{N}^\beta(\x,\scrV):= \sharp \big\{\y\in\scrV \colon |x-y|^d\leq u_x^{-\beta}\big\}. 
\]
The graph \(\G\) is then generated, given \(\scrV\), by connecting \(\x=(x,u_x)\) and \(\y=(y,u_y)\) with probability
\begin{align}\label{eq:BoolInferences}
   \mathbf{p}(\x,\y,\scrV\setminus\{\x,\y\}) = \1_{\{u_x<u_y\}}\frac{1\wedge u_x^{-\gamma\delta}|x-y|^{-d\delta}}{1+\mathscr{N}^{\beta}(\x,\scrV\setminus \{\x,\y\})}+\1_{\{u_x\geq u_y\}}\frac{1\wedge u_y^{-\gamma\delta}|x-y|^{-d\delta}}{1+\mathscr{N}^{\beta}(\y,\scrV\setminus \{\x,\y\})}.
\end{align} 

\begin{corollary}[Soft Boolean model with local interference] \label{corol:BoolInterferences}
	Let \(\gamma<1\), \(\delta>1\) and \(\beta\in[0,1)\). The soft Boolean model with local interference~\eqref{eq:BoolInferences} has Property~\ref{G:P_mixing}\(_\lambda^{\xi}\) with parameter \(\xi=d(1-1/\beta)\) for all \(\lambda\). Further,
	\[
		\overline{\zeta}_\lambda \equiv \zeta = (2-\delta)\vee \tfrac{1-\delta(1-\gamma)-\beta}{\gamma\delta-\beta}. 	
	\] 
	In particular, \(\overline{\zeta}_\lambda<0\) if \(\delta>2\) and \(\gamma<(\delta+\beta-1)/\delta\). Thus, we have \(\wlc>0\) in this case, and for all \(\lambda<\wlc\), 
	\[
		\limsup_{r\to \infty}\frac{\log \P_\lambda\big(B(r)\xleftrightarrow[]{}B(2r)^\textsf{c}\big)}{\log r} \leq d\big((\delta-2)\vee \tfrac{1-\delta(1-\gamma)-\beta}{\gamma\delta-\beta} \vee (1-1/\beta)\big).
	\]
\end{corollary}

If we choose \(\beta=0\), each sphere of interference has radius one. Hence, the model reduces to a version of the soft Boolean model with additional yet ultimately insignificant fluctuations on large scales. Indeed, for \(\beta=0\) the identified \(\zeta<0\) phase aligns with the results found for the soft Boolean model. For a phase diagram illustrating the \(\zeta<0\) phase of the model with local interference see Figure~\ref{fig:BoolInterferences}.
	
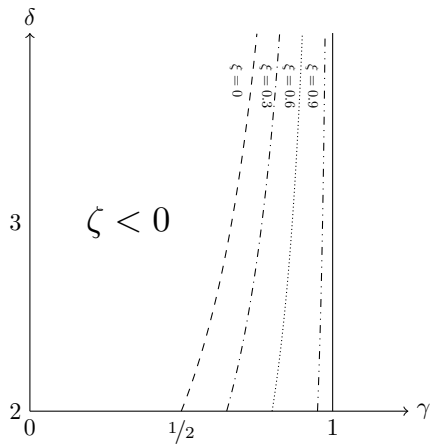
\begin{figure}[t!]
    \begin{center}
    \resizebox{0.4\textwidth}{!}{
    \begin{tikzpicture}[every node/.style={scale=0.9}]
        \draw[->] (0,0) -- (5,0) node[right] {$\gamma$};
        \draw	(0,0) node[anchor=north] {0}
        (2,0) node[anchor=north] {\nicefrac{1}{2}}
        (4,0) node[anchor=north] {1};

        \draw[->] (0,0) -- (0,5) node[above] {$\delta$};
        \draw (0,0) node[anchor=east] {2}
            (0,2.5) node[anchor=east] {3};

        \draw[] (4,0) -- (4,5);

        \draw (1.3,2.5) node[scale = 1.5, thick]{$\zeta<0$};
        
        \draw[dashed] (0,0) plot[domain=0.5:3/4,variable=\g]({4*\g},{-5+2.5*1/(1-\g)});
        \draw (90/33, 4.3) node[scale = 0.6, rotate = 270] {$\xi =0\phantom{.0}$};

        \draw[dash dot] (0,0) plot[domain=13/20:33/40,variable=\g]({4*\g},{-5+2.5*((1-0.3)/(1-\g)});
        \draw (103/33, 4.3) node[scale = 0.6, rotate = 270] {$\xi =0.3$};

        \draw[densely dotted] (0,0) plot[domain=16/20:36/40,variable=\g]({4*\g},{-5+2.5*((1-0.6)/(1-\g)});
        \draw (113/33, 4.3) node[scale = 0.6, rotate = 270] {$\xi =0.6$};

        \draw[dash dot dot] (0,0) plot[domain=19/20:39/40,variable=\g]({4*\g},{-5+2.5*((1-0.9)/(1-\g)});
        \draw (123/33, 4.3) node[scale = 0.6, rotate = 270] {$\xi =0.9$};
        
    \end{tikzpicture}
    }
    \end{center}
    \caption{Phase diagram for \(\gamma\) and \(\delta\) for the soft Boolean model with local interference. Represented from left to right the phase transition for \(\zeta<0\) for \(\xi=0,0.3,0.6,0.9\).} 
    \label{fig:BoolInterferences}
\end{figure}


\section{Further examples}\label{sec:correlatedExamples}
We briefly discuss additional examples and generalisations of the previous section. We do not examine these models in as much detail as the WDRCMs but illustrate, by example, how the derived results can be applied to models that involve more than one vertex weight or are constructed on more correlated point sets.

\subsection{Ellipses percolation} \label{sec:ellipsis}
Ellipses percolation is a percolation model first introduced in 2017~\cite{Teixeira_Ungaretti_2017} as a generalisation of the classical Boolean model with random radii. The vertices are given through a standard Poisson point process of intensity \(\lambda>0\) in the plane. Each vertex is centre of an ellipse with a uniform direction, minor axis equal to one and a heavy-tailed major axis. More precisely, each vertex \(\x\) is assigned an independent direction \(\pi_x\), distributed uniformly on \([-\pi/2,\pi/2)\), and another independent vertex mark \(R_x\), distributed according to a Pareto distribution with parameter \(2/\gamma>0\). Here, the parametrisation is chosen to be consistent with that above. Each vertex \(\x\) it thus associated with the ellipse, centred in \(x\in\R^2\) with minor axis of length \(1\), major axis of length \(R_x\), and orientation \(\pi_x\). We connect any pair of vertices by an edge if their associated ellipses intersect. This model can be seen as a generalisation of the standard Boolean model where the circles around each vertex are replaced by the described ellipses. However, it is worth noting that using ellipses introduces additional correlations into the model. These additional correlations have interesting effects, some of which we comment on below. Note that the model is translation invariant and takes place in the independent setting of Definition~\ref{def:IndS}. The model is locally finite for all \(\gamma<2\)~\cite{Teixeira_Ungaretti_2017}, and monotone in the sense of~\eqref{G:monotone}.

\begin{corollary}[Ellipses percolation]\label{corol:ellipsis} 
	Consider ellipses percolation described above for \(\gamma>0\) and \(\lambda>0\).
	\begin{enumerate}[(i)]
		\item If \(\gamma<1\), then \(\wlc>0\) and for all \(\lambda<\wlc\), we have for an appropriate constant \(C>0\),
			\[
				\P_\lambda(\cL(r,1))\leq \lambda C r^{2(1-1/\gamma)}
			\]
			as well as
			\[
			 \lim_{r\to\infty}\frac{\log \P_\lambda\big(B(r)\xleftrightarrow[]{}B(2r)^\textsf{c}\big)}{\log \P_\lambda(\cL(r,1))}=1.
			\]
		\item If \(\gamma>1\), then \(\wlc=0\).
	\end{enumerate} 
\end{corollary}

The proof can be found in Section~\ref{sec:proofsExamples}. We briefly compare our result with the findings in~\cite{Teixeira_Ungaretti_2017}. It is shown in the Theorems~1.1.2 and~1.2.2 therein that the model is locally finite but contains no subcritical percolation phase for \(1<\gamma<2\). This aligns with our result, as \(\lc=0\) implies \(\wlc=0\). For the standard Boolean model it is impossible to be locally finite while having \(\lc=0\) at the same time. Only if additional long-range edges are introduced, such a behaviour can be achieved~\cite{GLM2021}. This implies that some sort of long-range behaviour is entailed in ellipses percolation. Indeed, in~\cite{Hilario-Ungaretti-2021} a comparison between ellipses percolation and long-range percolation is used to derive chemical distances in the \(1<\gamma<2\) regime. Conversely, our result states \(\wlc>0\) and thus implies \(\lc>0\) for \(\gamma<1\), which also aligns with the findings in~\cite{Teixeira_Ungaretti_2017}. Hence, except potentially for the boundary regime \(\gamma=1\), we find that the two critical intensities are either both positive or both zero. This is unlike the other long-range models studied in the previous section, all of which contained a parameter regime with \(\wlc=0<\lc\). The question of whether a Boolean model with arbitrary convex grains can have \(\lc=0\) while still being locally finite is explored in detail in the preprints~\cite{GKM2024_robustness,gracar2025chemical}. The results therein include but are not limited to generalisations of ellipses percolation to higher dimensions together with additional correlations between the axes. Whether \(\lc>0\) implies \(\wlc>0\) in all these models remains an interesting open question that we leave for future work.

\subsection{Models based on Cox processes}\label{sec:Cox}
So far we have exclusively presented applications of our results towards models where the underlying point cloud is given by Poisson point processes. However, our mixing Condition~\ref{G:mixing}\(_\lambda\) not only allows us to cover models where the edge-drawing mechanism is not strictly local, as seen in Section~\ref{sec:BoolInterferences}, but models based on correlated point processes as well. For simplicity, we focus here on correlations introduced by the vertices, assuming independent uniform markings as well as translation-invariant connection probabilities of the form $\mathbf{p}(\x,\y)$ with 
\begin{equation} \label{eq:CoxLocFinite}
	\int \d u_o\int\d x\int \d u_x\ \mathbf{p}((o,u_o),(x,u_x))<\infty.
\end{equation}
That is, only the end vertices of a potential edge are relevant in a translation invariant and integrable way. Put differently, we consider here the WDRCM based on more general point clouds. Let us however note that of course also mixed versions can be considered where correlations are present in both the vertex positions and the edge drawing mechanism. 

To enable explicit computations, we further narrow our attention here to Cox point processes, i.e.\ Poisson point processes with a random intensity. Let us mention however that also other classes of point processes such as many Gibbs point processes satisfy Condition~\ref{G:mixing}\(_\lambda\), as well. Their analysis in this context is left for future work and we refer the reader to the associated literature, respectively. In this section, \(\scrV\) is a (independent and uniformly marked) Cox point process, meaning a Poisson point process on $\R^d$ with intensity measure $\lambda\Lambda$, where $\lambda>0$ and $\Lambda$ is a random element of the space of Borel measures on $\R^d$, equipped with the evaluation \(\sigma\)-algebra. We assume that the distribution of $\Lambda$ is invariant under spatial shifts and that $\E[\Lambda([0,1]^d)]=1$. The annealed process (with respect to the random environment $\Lambda$) is then stationary with intensity $\lambda$ and simple if $\Lambda$ contains no atoms almost surely. 

An interesting class of examples can be constructed for $d=2$, where the directing measure $\Lambda$ is given by the edges of a random line process. More precisely, let $S\subset \R^2$ be a random segment process, then $\Lambda(A)=\nu_1(S\cap A)$, for all Borel sets $A\subset\R^d$, where $\nu_1$ denotes the one-dimensional Hausdorff measure. The segment process can be given, for example, by the edges of an independent Poisson--Voronoi tessellation or the rectangular line process, where lines are drawn perpendicular to two independent Poisson processes on the $x$- and $y$-axis in $\R^2$. Let us highlight, that, due to the local nature of the construction of the Poisson--Voronoi tessellation, the associated Cox point process features (exponential) decorrelations, whereas the Cox point process based on the rectangular line process does not possess this property (at least not along the axes). Alternatively, a large class of examples can be constructed, where the directing measure is given by a random density with respect to the Lebesgue measure, i.e.\ $\Lambda(\d x)=\ell_x\d x$. For instance, the density $(\ell_x)_{x\in \R^d}$ can be defined as $\ell_x=\phi_1\1\{x\in \Xi\}+\phi_2\1\{x\not\in \Xi\}$, where $\Xi\subset\R^d$ is a random closed set and $\phi_1,\phi_2\ge 0$. In turn, $\Xi$ could be, for example, given by an independent Poisson--Boolean model, with or without random radii (viewed as a subset of \(\R^d\), not as a graph). 

Let us mention that the percolation analysis for Boolean models based on Cox point processes has seen some interest in recent years~\cite{hirsch2022sharp,jahnel2022phase,hirsch2019continuum,jahnel2023continuum}, with results that essentially recover the nontrivial percolation behaviour seen in the classical Poisson case. In this work, we extend this analysis by considering not only Gilbert graphs based on Cox point processes, but also more general weight-dependent graph structures, at least with respect to the subcritical annulus-crossing phase. We do not consider concrete connection mechanisms or deal with the Property~\ref{G:noLongEdge}\(_\lambda\) here but solely focus on the mixing condition as many possibilities for \(\mathbf{p}\) have already been discussed in Section~\ref{sec:wdrcm}.

We now present an example of a non-trivial Cox point process featuring polynomial mixing in $\R^2$. For this, consider $\Lambda(\d x)=\ell_x\d x$ with 
\begin{align}\label{eq:Cox}
\ell_x=\phi\sum_{i\in I}\1\{|Y_i-x|<\rho_i\},
\end{align}
where $\phi>0$ and $\Y=((Y_i,\rho_i):i\in\N)$ is an independent homogeneous Poisson point process with intensity one and i.i.d.~marks $\rho_i\ge 0$, distributed according to a Pareto distribution with tail index $\beta> 2$.  
In words, $\ell_x$ is the number of Poisson--Boolean discs that overlap a location $x$ and, since $\beta>2$, the parameter $\phi=\phi(\beta)$ can be tuned in order for $\Lambda$ to be normalised in the sense $\E[\ell_o]=1$. Then, we have the following result, whose proof can be found in Section~\ref{sec:proofsExamples}:

\begin{prop}[Cox percolation]\label{prop:Cox} 
The model defined via~\eqref{eq:Cox} with the assumptions~\eqref{eq:CoxLocFinite} on the edge-drawing mechanism satisfies the Conditions~\ref{G:Point_process} -- \ref{G:locallyFinite} and is \emph{monotone} in the sense of~\eqref{G:monotone}. Moreover, the model has the \emph{polynomially mixing Property}~\ref{G:P_mixing}\(_\lambda^\xi\) for $\xi=2-\beta$ and all \(\lambda>0\). 
\end{prop}

\subsection{Models based on worm percolation}\label{sec:FRI}    
After primarily focusing on continuum models, we now apply our results to an example of a lattice-based model. While the case of independent Bernoulli site-percolated lattices is similar to the Poisson process (cf.\ Remark~\ref{rem:proofMain}), the objective here is to discuss a correlated lattice model based on \emph{percolation of worms}~\cite{rath_rokob_2022_worm}. 

Before diving deeper into this model and its mixing properties, let us comment on another important model for a correlated lattice, namely \emph{random interlacements} (RI)~\cite{Sznitman_2010_RI}; we refer the reader to~\cite{DrewitzEtAll_2014_RI} for an introduction to RI. In a nutshell, this model is a Poisson soup of doubly infinite random walk trajectories in \(\Z^d\), \(d\geq 3\), with a parameter \(u>0\) determining the intensity of these trajectories. The vertex set \(\scrV\) consists of all sites visited by at least one trajectory. When equipped with nearest-neighbour edges, it is obvious that a single trajectory already forms an infinite component and crosses all annuli by transience of the simple random walk in dimensions three and higher. Moreover, it can even be shown that the corresponding graph is connected for all intensities~\cite{Sznitman_2010_RI, CernyPopov2012}. Consequently, our results cannot apply for RI. Indeed, \emph{our} mixing assumption is never fulfilled as any random-walk trajectory that crosses an annulus with radii \(r\) and \(2r\) has a significant chance to also cross another such annulus at distance \(r\). Note, that this is a result of the two annuli and their distance growing at the same rate. In a more standard notion, however, RI is mixing as, for two given vertices, their covariance is determined by the Green function which decays polynomially with exponent \(2-d\) in dimensions \(d\geq 3 \). This shows the importance of our stronger mixing assumption in our results.

However, the RI model can be made mixing in our sense if the length of the random walks is limited, which brings us to percolation of worms. The model is defined as follows. Fix a \emph{step-length distribution} \(\ell\) and an intensity parameter \(\nu>0\). We independently assign each site of \(\Z^d\) a Poisson number of worms with mean \(\nu\). Each worm placed at \(x\in\Z^d\) draws an independent copy of \(\ell\) and performs the corresponding number of simple random walk steps, starting from \(x\). The vertex set \(\scrV\) then consists of all sites of \(\Z^d\) visited by at least one worm. Classically, the graph \(\scrG\) is then constructed by adding all nearest-neighbour edges between vertices of \(\scrV\). Just like in the previous Section~\ref{sec:Cox}, we stick with this choice and focus primarily on the mixing properties but, in principle, other connection mechanisms, e.g.\ those of Section~\ref{sec:wdrcm}, could be considered. Observe that \(\scrV\) is stationary and the associated nearest-neighbour graph is translation invariant and locally finite.

An important special case for the step-length distribution \(\ell\) is the geometric distribution with mean \(T\geq 1\)~\cite{Lawler_2013_IRW}. If we further restrict to \(d\geq 3\), we obtain \emph{finitary random interlacements} (FRI) with intensity \(u:=(T+1)/(2d)\) and walk length \(T\), as introduced in~\cite{Bowen2019}. Notably, FRI converge in distribution to RI with intensity \(2du\) as \(T\to\infty\), giving the model its name. While RI always contain a unique infinite component for all \(u>0\), this holds no longer true for FRI. Intriguingly, there are parameter regimes where there are infinitely many infinite components~\cite{Bowen2019, Procaccia_etAl_2021}.

In our setting, we are not restricted to geometric step-length distributions but may consider general distributions \(\ell\) and study the positivity of \(\wlc\). Fix such a distribution \(\ell\) and recall from Remark~\ref{rem:Assumptions} that the intensity parameter our results apply for is defined by \(\lambda=\P(o\in\scrV)\). That is, \(\lambda\) is given by the probability that the origin is visited by at least one worm. Note that the \emph{expected} number of walks that traverse \(o\) is given by \(\nu\E[\ell] \P_o^{\ell}(\tilde{H}_o=\infty)\), where \(\P_o^{\ell}(\tilde{H}_o=\infty)\) denotes the probability that a worm with step-length distribution \(\ell\) starting in \(o\) never returns, see~\cite{DrewitzEtAll_2014_RI,Bowen2019}. Further, \(\lambda\) is bounded from below by \(1-e^{-\nu}\), the probability that at least one worm has been assigned to \(o\) in the beginning. Therefore, \(\nu\sim 1-e^{-\nu}\leq \lambda\leq C\nu\) and \(\lambda\) can be appropriately varied by adjusting \(\nu\) whenever \(\ell\) has finite expectation, which we assume in the following.

It is shown in~\cite{rath_rokob_2022_worm} that \(\lc>0\) if \(\ell\) has finite variance in all dimensions and \(\lc=0\) in dimensions \(\d\geq 5\) if \(\ell\) has infinite second moment and a slightly stronger tail assumption on \(\ell\) is satisfied. For annulus-crossings, however, the situation is quite different. While the results of~\cite{rath_rokob_2022_worm} are independent of the dimension (provided \(d\geq 5\)), our findings depend on the dimension; a characteristic that has not previously be shown in our analysis.

\begin{corollary}[Percolation of worms] \label{corol:FRI}
	Consider percolation of worms with integrable step-length distribution \(\ell\) and nearest-neighbour connection rule.
	\begin{enumerate}[(i)]
		\item If \(\ell\) has finite \(d\)-th moment, i.e., \(\sum_{n\in\N} n^d \ell\{n\}<\infty\), then \(\wlc>0\), and
		\item if \(\ell\) has infinite \(d\)-th moment, i.e., \(\sum_{n\in\N} n^d \ell\{n\}=\infty\), then \(\wlc=0\).
	\end{enumerate}
\end{corollary}

\subsection[\textit{k}-nearest-neighbour percolation and other intensity-independent models]{\textit{k}-nearest-neighbour percolation and other $\boldsymbol{\lambda}$-independent models} \label{sec:kNN}
One may question whether the assumption~\ref{G:noLongEdgeUnif}\(_\lambda\) (resp.~\ref{G:mixing}\(_\lambda\)) could be relaxed and replaced by the assumption \(\{\forall \lambda'\leq \lambda,\) \ref{G:noLongEdge}\(_\lambda\}\) (resp.~\ref{G:lambda-mix}\(_\lambda\)) in the statement of Theorem~\ref{thm:main}. We show that this is not possible by discussing a simple class of models where the $\lambda$-dependence of the model is only superficial. Consider any graph $\scrG_1=(\scrV_1,\scrE_1)$ on $\R^d$ obeying our settings, but defined only for $\lambda=1$. Then, one may (re)define the model for $\lambda\ne 1$ by keeping the exact same graph structure, but scaling all vertex locations by a factor $\lambda^{-1/d}$. That is,
\[
\scrV_\lambda:=\left\{\lambda^{-1/d} x\colon x\in \scrV_1\right\}, \qquad \scrE_\lambda:=\left\{\{\lambda^{-1/d} x, \lambda^{-1/d}y\}\colon \{x,y\}\in \scrE_1\right\}.
\]
This produces a model \((\scrG_\lambda:\lambda>0)\) that is perfectly legitimate in our general settings. However, a phase-transition in \(\lambda\) is clearly impossible. For example, let $\scrG_1$ be a classical WDRCM with no long edges (i.e., \ref{G:noLongEdge}\(_1\) is satisfied) but for which the annuli are crossed (still for $\lambda=1$). With our definition of $\G_\lambda$, we immediately infer the Properties~\ref{G:noLongEdge}\(_\lambda\) and \ref{G:mixing}\(_\lambda\) for all values of $\lambda$ (for mixing, note the model satisfies the independent settings condition and thus has no relevant correlations), but still $\wlc=0$. It is clear that this model does neither satisfy Monotonicity~\eqref{G:monotone} nor~\ref{G:noLongEdgeUnif}\(_\lambda\). Another interesting example in this class of models is \emph{\(k\)-nearest-neighbour percolation}~\cite{BalisterBollobas_2013}, which contains an infinite outgoing component if \(k\) is large enough, independently of \(\lambda\). This model also satisfies the  Properties~\ref{G:noLongEdge}\(_\lambda\) and~\ref{G:lambda-mix}\(_\lambda\) for every $\lambda>0$, yet it neither satisfies ~\ref{G:noLongEdgeUnif}\(_\lambda\) nor~\ref{G:mixing}\(_\lambda\) for any $\lambda>0$.

	
\section{Proofs}\label{sec:proofs}
	\subsection{Proofs of the results in Section~\ref{sec:mainResults}} \label{sec:proofsMainResults}
	We give the proofs of our main Theorems~\ref{thm:main},~\ref{thm:equiStatement}, and~\ref{thm:Diamter} in this section. We start with the proof of Theorem~\ref{thm:main} as the other results rest on either its proof or its statement. To this end, we adapt a renormalisation technique developed by Gou\'{e}r\'{e} for the Poisson--Boolean model~\cite{Gouere08}.
		
	\begin{proof}[Proof of Theorem~\ref{thm:main}.]
		Let us start with the proof of Part~(ii). Fix any \(\lambda>0\). Then, as \ref{G:noLongEdge}\(_\lambda\) is not satisfied, we have \(\P(\cL(r,3))>c\) for all \(r\geq 1\). The event \(\cL(r,3)\) implies the existence of an edge connecting a vertex in \(B(r)\) to some vertex at distance at least \(3r\). The latter vertex is located outside \(B(2r)\) by necessity. Therefore 
		\[
			\P_{\lambda}\big(B(r)\xleftrightarrow[]{}B(2r)^\textsf{c}\big)\geq \P_\lambda(\cL(r,3))>c
		\] 
		and thus \(\lambda>\wlc\). As \(\lambda\) can be chosen arbitrarily small, this yields \(\wlc=0\), proving Part~(ii).
		
		In order to prove Part~(i) of the theorem, we have to show that \(\wlc>0\) whenever \(\scrG\) fulfils the Properties~\ref{G:mixing}\(_{\bar\lambda}\) and~\ref{G:noLongEdgeUnif}\(_{\bar\lambda}\) for some \(\bar\lambda>0\). Let us denote the annulus-crossing event by
		\begin{equation}\label{eq:eventAnnulusCross}
			\cE(x,r) = \big\{B(x,r)\xleftrightarrow[]{} B(x,2r)^{\textsf{c}}\big\} 
		\end{equation}
		and abbreviate \(\cE(r)=\cE(o,r)\). Then the claim is equivalent to showing \(\P_\lambda(\cE(r))\to 0\), as \(r\to\infty\), for \(\lambda\) sufficiently small. To this end, we adapt the renormalisation scheme of~\cite{Gouere08}. Let us introduce two additional events. First, we let
		\begin{equation*}
			\cF(r) := \{\exists \x\sim \y\colon |x|<2r, |x-y|>r/10\} = \cL(2r,1/20),
		\end{equation*}
		be the event that there exists an edge longer than \(r/10\) incident to some vertex located in \(B(2r)\). Secondly, we recall the local annulus-crossing event \(\cG(x,r)\) with \(\cG(r)=\cG(o,r)\), introduced in~\eqref{eq:local_annulus_cross}. In words, the event \(\cE(r)\) defines the annulus crossing subject to our main result. The event \(\cF(r)\) describes the presence of long edges in the considered annulus. In particular, if \(\cF(r)\) is not fulfilled, any annulus crossing cannot use far apart vertices. The event \(\cG(r)\) can then be seen as a localised variant of \(\cE(r)\) when no such far apart vertex is used and any path leaving \(B(2r)\) thus ends in \(B(3r)\setminus B(2r)\). Naturally, \(\cG(r)\subset\cE(r)\). The main step in the renormalisation is then to argue that either long edges in the sense of \(\cF(r)\) are present or the localised annulus crossing event occurs twice almost independently on a smaller scale due to the mixing property. To make this precise, let us denote by \(\partial B \) the boundary of a set \(B\subset\R^d\). Let further \(I,J\subset\R^d\) be two finite sets satisfying \(I\subset\partial B(10)\) and \(J\subset\partial B(20)\) such that 
		\[
			\partial B(10)\subset\bigcup_{i\in I}B(i,1/2), \quad \text{ and } \quad \partial B(20)\subset\bigcup_{j\in J} B(j,1/2).
		\]
        Put differently, \(I\) and \(J\) induce a covering of the spheres \(\partial B(10)\) and \(\partial B(20)\), respectively, with balls of radius $1/2$. Note that this particularly implies that the annulus \(B(10.5)\setminus B(9.5)\) is completely covered by \(\bigcup_{i\in I} B(i,1)\), that is, using radius-\(1\) balls instead of radius-\(1/2\) balls. This covering property particularly implies that any vertex located in said annulus is at distance at most \(1\) to some point in \(I\). The same is true for \(B(20.5)\setminus B(19.5)\) and the union \(\bigcup_{j\in J}B(j,1)\). The outlined key observation is then 
        \begin{equation}\label{eq:FactoriseP(G)}
	    	\cE({10 r})\setminus \cF(10r)\subset \Big(\bigcup_{i\in I}\cG(r\cdot i,r )\Big)\cap\Big(\bigcup_{j\in J}\cG(r\cdot j,r)\Big),
		\end{equation}
		where we use the notation \(r\cdot i\) for the product of \(r\) and \(i\) to put emphasis on the fact that \(i\) is an element of \(\R^d\) while \(r\in\R\). Equation~\eqref{eq:FactoriseP(G)} follows from the fact that on \(\cE({10 r})\setminus \cF(10r)\), there exists a path from a vertex located in \(B(10r)\) to some vertex located outside of \(B(20r)\) using no edges longer than \(r\). 
        Therefore, at least one vertex of said path must lie within distance \(r/2\) to the sphere \(\partial B(10 r)\), and thus within distance \(r\) to some point in $r\cdot I$. Put differently, this vertex is located in \(B(r\cdot i, r)\) for some $i\in I$. Restarting the path from this vertex, it realises the localised annulus-crossing event \(\cG(r\cdot i, r)\), as the original path must ultimately reach a vertex located in \(B(20 r)^\textsf{c}\) using only edges shorter than \(r\), and, in particular, it must first reach some vertex located in \(B(r\cdot i, 3r)\setminus B(r\cdot i, 2r)\) passing only vertices located in \(B(r\cdot i, 3r)\). Again, we have used that all edges are shorter than \(r\). By a similar argumentation also \(\cG(r\cdot j,r)\) occurs for some \(j\in J\), justifying~\eqref{eq:FactoriseP(G)}. The above argument is sketched in Figure~\ref{fig:sketch}. 
		\begin{figure}[t!]
    		\begin{center}
    			\resizebox{1\textwidth}{!}{
   				\begin{minipage}{\textwidth}
    				\centering
    				\begin{tikzpicture}[scale = 0.5, every node/.style={scale=0.4}]
        				\clip (-2,1.5) rectangle ++(18,-15);
        				\draw[] (0,0) circle(4);
        				\draw[] (0,0) circle(4.5);
        				\draw[] (0,0) circle(14);
        				\draw[] (0,0) circle(3.5);
        				\node (S) at (2.6,-2.65)[scale = 0.5, circle, fill = red, color =red,  label = {},draw]{};
        				\node (T1) at (8, -7.5)[draw, scale = 0.1]{};
        				\node (T2) at (11.8,-6)[draw, circle, scale = 0.5, color =blue, fill = blue]{};
        				\node (T3) at (14,-7)[draw, circle, scale = 0.3, color =black, fill = black]{};
        				\draw[decorate, decoration = {random steps, segment length = 1 mm}] (S)--(T1);
        				\draw[decorate, decoration = {random steps, segment length = 1 mm}] (T1)--(T2);
        				\draw[decorate, decoration = {random steps, segment length = 1 mm}] (T2)--(T3);
        				\draw (3.53,-1.9) circle(1);
        				\draw (2.19,-3.348) circle(1);
        				\draw (3.881,-0.968) circle(1);
        				\draw (4,0.025) circle(1);
        				\draw (1.3,-3.786) circle(1);
        				\draw (0.31,-3.99) circle(1);
        				\draw[red] (2.95,-2.7) circle(1);
        				\draw[red] (2.95,-2.7) circle(2);
        				\draw[red] (2.95,-2.7) circle (3);
       		 			\draw (12.874,-5.5) circle(1);
        				\draw (11.96,-7.277) circle(1);
        				\draw (11.41,-8.11) circle(1);
        				\draw (10.804,-8.904) circle(1);
        				\draw (13.23,-4.567) circle(1);
        				\draw (13.526,-3.611) circle(1);
        				\draw[blue] (12.445,-6.405) circle(1);
        				\draw[blue] (12.445,-6.405) circle(2);
        				\draw[blue] (12.445,-6.405) circle(3);
        				\draw (-1.3,-3.25) node[anchor = north, scale = 1.1, rotate = 350]{$\partial B(9.5 r)$};
        				\draw (-1.3,-3.8) node[anchor = north, scale = 1.1, rotate = 350]{$\partial B(10 r)$};
        				\draw (-1.3,-4.4) node[anchor = north, scale = 1.1, rotate = 350]{$\partial B(10.5 r)$};
        				\draw (13,1.5) node[anchor = north, scale = 1.1]{$\partial  B(20r)$};
    				\end{tikzpicture}
    			\end{minipage}
    			}
    		\end{center}
    		\caption{Sketch of~\eqref{eq:FactoriseP(G)} in \(d=2\). A path starting inside \(B(10r)\) and leaving \(B(20r)\), where no edge longer than \(r\) is used. The red vertex on the path is close to the center of the red ball of the covering of \(\partial B(10r)\) and it is the starting point for the event \(\cG(\)red balls\()\). Further, the covering of the annulus \(B(10.5 r)\setminus B(9.5r)\) is indicated. The same applies to the blue vertex on the path, which lies close to the center of the blue ball of the covering of \(\partial B(20r)\).}
    		\label{fig:sketch}
		\end{figure}
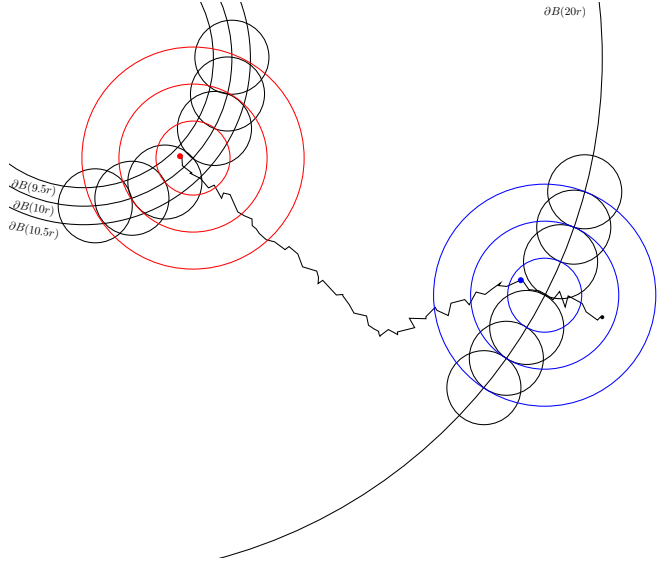
       	As a result, we obtain
		\[
			\P_\lambda\big(\cE(10 r)\setminus \cF(10 r)\big)\leq \sum_{i\in I, \, j\in J}\P_\lambda\big(\cG(r\cdot i,r)\cap \cG(r\cdot j,r)\big).
		\]
		Using translation invariance~\ref{G:Translation}, we infer, writing \(C_1=(\sharp I)(\sharp J)\), 
		\begin{equation*}
			\begin{aligned}
				\sum_{i\in I, \, j\in J} \P_\lambda\big(\cG(r\cdot i,r)\cap \cG(r\cdot j,r)\big) 
				& 
				\leq C_1 \max_{i\in I, j\in J}\Big(\P_\lambda\big(\cG(r\cdot i,r)\cap \cG(r\cdot j,r)\big)\Big) 
				\\ & \leq C_1\max_{i\in I, j\in J}\big(\operatorname{Cov}_\lambda(\1_{\cG(o,r)}, \1_{\cG(r\cdot (j-i), r)})\big) + C_1 \P_\lambda(\cG(r))^2.
			\end{aligned}
		\end{equation*}	
		Now, as \(I\) and \(J\) are finite sets, we can use the mixing Property~\ref{G:mixing}\(_{\bar\lambda}\) and write
		\begin{equation}\label{eq:mixingError}
			\operatorname{cov}(r) := C_1\sup_{\lambda<\bar\lambda}\max_{i\in I, j\in J}\big(\operatorname{Cov}(\1_{\cG(o,r)},\1_{\cG(r\cdot (j-i), r)})\big) = o(1).
		\end{equation}
		Using \(\cG(r)\subset\cE(r)\), we hence obtain
		\begin{equation} \label{eq:renorm}
			\begin{aligned}
				\P_\lambda(\cE(10r)) 
				& 
				\leq \P_{\lambda}(\cF(10r)) +\operatorname{cov}(r) + C_1 \P_\lambda(\cE(r))^2,
			\end{aligned}
		\end{equation}
		which is key in showing \(\P_{\lambda}(\cE(r))\to 0\) for sufficiently small \(\lambda\). Recall that we work under the Assumption~\ref{G:noLongEdgeUnif}\(_{\bar\lambda}\) for some \(\bar\lambda>0\). Hence, \(\P_\lambda(\cF(r))\to 0\) uniformly for all \(\lambda<\bar\lambda\). Therefore, there exists \(r_0> 1\) such that 
		\[
			C_1\big(\sup_{\lambda<\bar\lambda}\P_\lambda(\cF(r))+\operatorname{cov}(r)\big) \leq 1/4, \qquad \text{ for all }r\geq r_0.
		\]
		Let us define the following two functions on \([r_0,\infty)\):
		\begin{equation}\label{eq:def_f_g}
			g(r):= C_1\big(\sup_{\lambda<\bar\lambda}\P_\lambda(\cF(r))+\operatorname{cov}(r)\big) \quad \text{ and } \quad f_\lambda(r) = C_1 \P_{\lambda}(\cE(r)),
		\end{equation}
		and choose
		\begin{equation}\label{eq:lambdaCondition}
			\lambda < \bar\lambda \wedge \tfrac{1}{2\, C_1 \, (10 r_0)^d \omega_d},
		\end{equation}
		where \(\omega_d\) denotes the volume of the \(d\)-dimensional unit ball. Since \(\lambda < \bar\lambda\), we have \(g(r)\leq 1/4\). Further, since \(\lambda < 1/(2C_1(10 r_0)^d \omega_d)\), we have with Markov's inequality
		\begin{equation}\label{eq:useOfIntensity}
			f_\lambda(r) = C_1\P_\lambda(\cE(r))\leq C_1\E_\lambda[\sharp(\scrV \cap B(r))]=C_1 \, \lambda \; \omega_d \,  r^d \leq 1/2,   
		\end{equation}
		for all \(r\in[r_0,10 r_0]\), where the second to last step follows from the fact that the intensity measure of the underlying point process is the Lebesgue measure multiplied by \(\lambda\). Additionally,~\eqref{eq:renorm} implies
		\begin{equation}\label{eq:renorm_f_g}
			f_\lambda(r) \leq g(r) + f_\lambda(r/10)^2 \quad \text{ for all }r\geq 10 r_0. 
		\end{equation}
		Together with the established bounds on \(f_\lambda\) and \(g\), we infer \(f_\lambda(r)\leq 1/2\) for all \(r\geq r_0\). This can for instance easily be seen by proving inductively \(f_\lambda(r)\leq 1/2\) for all \(r\in[10^{n-1}r_0,10^n r_0]\). Applying the bound \(f_\lambda\leq 1/2\) as well as~\eqref{eq:renorm_f_g} iteratively yields for all \(n\in\N\) and \(r\in[r_0,10 r_0]\),
		\begin{equation*}
			\begin{aligned}
				f_\lambda(10^{n} r) & \leq g(10^{n}r)+ \frac{f_\lambda(10^{n-1}r)}{2} \leq \dots \leq \sum_{j=1}^n \frac{g(10^{j}r)}{2^{n-j}} + 2^{-n-1}.		
			\end{aligned}
		\end{equation*}
		Finally, for \(\varepsilon>0\), we can choose \(n\) large enough such that
		\begin{equation*}
			\begin{aligned}
				\sum_{j=1}^n \frac{g(10^{j}r)}{2^{n-j}} & \leq \frac{1}{4}\sum_{j=1}^{n/2}  2^{-n+j} + \big(\max_{n/2\leq k\leq n }g(10^k r)\big)\sum_{j=n/2}^n 2^{-n+j} \leq 2^{-n/2} + 2 \max_{n/2\leq k\leq n}g(10^k r)	\leq \varepsilon,		
				\end{aligned}
		\end{equation*}
		uniformly for \(r\in[r_0, 10r_0]\). Note here, that the \(\max g(10^k r)\) term on the right-hand side is small as the function \(g\) captures exactly the two assumptions of the theorem as it is defined via the corresponding events, cf.~\eqref{eq:def_f_g}. In conclusion, \(\P_\lambda(\cE(r))=f_\lambda(r)/C_1\to 0\) as \(r\to \infty\), for all \(\lambda\) satisfying~\eqref{eq:lambdaCondition}, proving Part~(i). This finishes the proof of Theorem~\ref{thm:main}.
		\end{proof} 
		
		\begin{remark}\label{rem:proofMain}
			We performed the above proof for the continuum point process case. If the vertex locations are given by a lattice subset that fulfils the assumptions of Remark~\ref{rem:Assumptions}, the proof can be performed in the same way. Firstly, the change of norm to \(||\cdot ||_\infty\) guaranties that the scaling properties of the sets \(I\) and \(J\) remain true. More precisely, points that induce a covering of the sphere \(\partial B(c)\) with radius \(1/2\) balls, also induce a covering of the sphere \(\partial B(c r)\) with radius \(r/2\) balls. Secondly, we use the given intensity measure to bound the function \(f_\lambda\) in~\eqref{eq:useOfIntensity}. However, as balls in \(\R^d\) and \(\Z^d\) have volume of the same order and only differ by a constant factor, the bound on \(f_\lambda\) remains true once the involved constants are adjusted by the definition of \(\lambda\) in Remark~\ref{rem:Assumptions}. 		
		\end{remark}
		
		We continue with the proof of Theorem~\ref{thm:Diamter} as it builds on the arguments derived in the previous proof and postpone the proof of Theorem~\ref{thm:equiStatement} to the end of the section.
		
		\begin{proof}[Proof of Theorem~\ref{thm:Diamter}.]
			Recall that we work under the assumption \ref{G:P_mixing}\(_\lambda^\xi\) for some \(\xi_\lambda<0\) and recall the decay exponent \(\zeta_\lambda<0\) describing the polynomial decay of \(\P_\lambda(\cL(r,1))\). As \(\P_{\lambda}(\cE(r))\) is a decreasing function in \(r\), the claim of Theorem~\ref{thm:Diamter} is equivalent to showing that \(r^{s-1}\P_{\lambda}(\cE(r))\) is integrable for all \(s<|\zeta_\lambda \vee \xi_\lambda|\), where \(\cE(r)\) denotes the annulus-crossing event as introduced in~\eqref{eq:eventAnnulusCross}. Note that the polynomial mixing condition can equivalently be written as 
			\[
				\big|\operatorname{Cov}_\lambda\big(\1_{\cG(r)},\1_{\cG(xr,r)}\big)\big| \leq \ell_\lambda(x,r) r^{\xi_\lambda}, 
			\]  
				where \(r\mapsto\ell_\lambda(x,r)\) is a slowly varying function~\cite{Bingham1987}. Hence, the mixing term in~\eqref{eq:mixingError} now reads
			\[
				\operatorname{cov}_\lambda(r) \leq \ell_\lambda(r) r^{\xi_\lambda}, \quad \text{ where } \quad \ell_\lambda(r)= \sum_{i\in I, j\in J} \ell(j-i, r).			
			\]
			Note that \(\ell(r)\) again is slowly varying. As above there exists \(r_0^{(\lambda)}\) such that the, now \(\lambda\)-dependent, function \(g_\lambda\) satisfies
			\[
				g_\lambda(r)= C_1\big(\P_\lambda(\cF(r))+\operatorname{cov}_\lambda(r)\big)<1/4
			\]
			on \([r_0^{(\lambda)},\infty)\). Recall further the function \(f_{\lambda}\) from the previous proof, cf.~\eqref{eq:def_f_g}. Essentially, we have reestablished the same setting as above but now for the fixed value of \(\lambda\) instead of taking the supremum over \(\lambda<\bar\lambda\). By definition of \(g_\lambda(r)\) and as \(\P_\lambda(\cF(r))\asymp\P_\lambda(\cL(r,1))\), we have that \(r^{s-1}g_\lambda(r)\) is integrable for all \(s<|\zeta_\lambda \vee  \xi_\lambda|\). To prove the claimed integrability of \(\P_\lambda(\cE(r))=f_\lambda(r)/C_1\), observe first that \(f_\lambda(r)=o(1)\) for \(\lambda<\wlc\). In particular, \(f_\lambda(r)\leq 10^{-s-1}/2\) for all \(r\geq R_\lambda/10\) for some sufficiently large \(R_\lambda\geq r_0^{(\lambda)}\). Following the arguments in~\cite{Gouere08}, we find for \(S>10 R_\lambda\), using~\eqref{eq:renorm_f_g},
			\begin{equation*}
				\begin{aligned}
					\int_{R_\lambda}^{S} r^{s-1} f_\lambda(r)\d r 
					& 
					\leq \int_{R_\lambda}^S r^{s-1}f_\lambda(r/10)^2\d r +\int_{R_\lambda}^{S} r^{s-1} g_\lambda(r) \d r
					\\ &
					\leq \tfrac{1}{2} \int_{R_\lambda/10}^{S/10} r^{s-1} f_\lambda(r) \d r + \int_{R_\lambda}^{S} r^{s-1} g_\lambda(r) \d r 
					\\ &
					\leq \tfrac{1}{2}\int_{R_\lambda/10}^{R_\lambda} r^{s-1} f_\lambda(r) \d r + \tfrac{1}{2}\int_{R_\lambda}^{S} r^{s-1}f_\lambda(r) \d r + \int_{R_\lambda}^\infty r^{s-1} g_\lambda(r) \d r,
				\end{aligned}
			\end{equation*}
			where we used the change of variables \(r\mapsto 10r\) together with the bound \(f_\lambda(r)\leq 10^{-s-1}/2\) in the second step. Sending \(S\to\infty\), the established integrability of \(r^{s-1}g_\lambda(r)\) then yields
			\begin{equation*}
				\begin{aligned}
					\tfrac{1}{2}\int_{R_\lambda}^\infty r^{s-1}f_\lambda(r) \leq \tfrac{1}{2}\int_{R_\lambda/10}^R r^{s-1} f_\lambda(r) \d r + \int_{R_\lambda}^\infty r^{s-1} g_\lambda(r) \d r <\infty,
				\end{aligned}
			\end{equation*}
			for all \(s<|\zeta_\lambda \vee \xi_\lambda|\). This concludes the proof of Theorem~\ref{thm:Diamter}.
		\end{proof}
		
		Next, we give the proof of Theorem~\ref{thm:equiStatement} stating that the rareness of long edges and the positivity of \(\wlc\) are equivalent in the independent setting.

		\begin{proof}[Proof of Theorem~\ref{thm:equiStatement}.]
			Recall that in the independent setting, and for \(\lambda<\lambda'\), the graph \(\scrG_\lambda\) can be constructed from \(\scrG_{\lambda'}\) by performing site percolation with retention parameter \(\lambda/\lambda'\), as explained below Definition~\ref{def:IndS}. Hence, we have
			\begin{equation*}
				\P_{\lambda}(\cL(r,1)) \le \P_{\lambda'}(\cL(r,1)).
			\end{equation*}
			as well as
			\begin{equation*}
                \begin{aligned}
                    \P_{\lambda}(\cL(r,1)) 
                    &= 
                        \P_{\lambda'}\big(\exists \x\sim \y\colon |x|<r, |x-y|>r, \x,\y \text{ survive percolation}\big) 
                    \\ &
                        \geq \P_{\lambda'}\big(\exists \x\sim \y\colon |x|<r, |x-y|>r, \ \substack{\text{of these edges the one with endpoint} \\ \text{closest to the origin survives percolation}}\big) 
                    \\ &
                        =\big(\tfrac{\lambda}{\lambda'}\big)^2\P_{\lambda'}(\cL(r,1)),   
                \end{aligned}
			\end{equation*}
			where an edge survives percolation if both of its endpoints do. Both inequalities combined prove the independence of~\ref{G:noLongEdge}\(_\lambda\) and~\ref{G:noLongEdgeUnif}\(_{\lambda}\) from \(\lambda\) and show that both probabilities have the same tail. The equivalence of~\ref{G:noLongEdgeUnif}\(_{1}\) and \(\wlc>0\) is then an immediate consequence of Theorem~\ref{thm:main}.
		\end{proof}
		
		
		\subsection{Proofs of the results in Sections~\ref{sec:wdrcm} and~\ref{sec:correlatedExamples}}\label{sec:proofsExamples}
		In this section we prove the results established for the WDRCM and all the other examples.

		\subsubsection*{Proofs for the WDRCM of Section~\ref{sec:deff}}
		Recall first the framework of the  WDRCM. In particular, recall that vertices \(\x=(x,u_x)\) are given via a homogeneous Poisson point process on \(\R^d\times(0,1)\) and recall the connection mechanism given in~\eqref{eq:classicPhi}. Key in proving the results for WDRCMs and, particularly, Theorem~\ref{thm:WDRCM} is to properly relate the tail behaviour of \(\P_{\lambda}(\cL(r,1))\) with the exponent \(\zeta\), derived in~\eqref{eq:zeta}. To this end, we need some auxiliary result about the involved vertex marks. Indeed, when we introduced \(\zeta\) heuristically, we claimed that the empirical distribution of the `weak' vertices in a set consisting of \(r^d\) vertices should roughly be \(r^d \1_{[r^{-d+d\zeta},1]}\) to argue that \(\zeta\) describes the occurrence of long edges appropriately. Before giving the proof of Theorem~\ref{thm:WDRCM}, we shall make this claim precise. 
        
        Recall that, given the vertices' locations, the associated vertex marks form an independent collection of uniformly distributed random variables. We are thus interested in some kind of regular behaviour of independent uniforms guaranteeing that the marks are evenly distributed over the interval.  To this end, consider a collection of independent random variables \(U=(U_1,\dots,U_n)\) distributed uniformly on \((0,1)\) and denote their joint distribution by \(P\). We follow the arguments of~\cite{GraLuMo2022} and fix a parameter \(0<a<1\) and define for all \(i=1,\dots,\lfloor n^{1-a}\rfloor\) 
		\[
			N^a_n(i) = \sum_{j=1}^n \1\Big\{U_j\leq\frac{i}{\lfloor n^{1-a}\rfloor}\Big\}.
		\]
		We say that \(U\) is \emph{\(a\)-regular} if
		\[
			N^a_n(i)\geq \frac{i n}{2\lfloor n^{1-a}\rfloor}, \quad \forall \, i=1,\dots,\lfloor n^{1-a}\rfloor.
		\]
        The following lemma shows that \(a\)-regularity provides the desired regularity of the vertex marks. 
        
        \begin{lemma} \label{lem:a-regular} ~\
            \begin{enumerate}[(i)]
                \item 
                    For the collection \(U=(U_1,\dots,U_n)\) of independent uniform distributions with joint distribution \(P\) and any \(a\in(0,1)\), we have
                    \[
                        P(U \text{ is } a\text{-regular}) \leq e^{-c n^a},
                    \]
                    for some constant \(c>0\) and \(n\in\N\) large enough.
                \item Consider an \(a\)-regular realisation \(u=(u_1,\dots,u_n)\) with \(n> 2^{1/(1-a)}\). Then, for any non-increasing function \(\varphi\colon [0,1]\to[0,\infty)\), we have 
                    \[
                        \sum_{j=1}^n \varphi(u_j)\geq \frac{n}{2}\int_{2n^{a-1}}^1 \varphi(t) \, \d t.
                    \]
            \end{enumerate}
            
        \end{lemma}
\begin{proof}
	A simple calculation yields
		\[
			E [N^a_n(i)] = \frac{in}{\lfloor n^{1-a}\rfloor},
		\]
	so that the standard Chernoff bound for independent Bernoulli random variables implies
	\[
		P\big(N^a_n(i)<\tfrac{1}{2}E[N^a_n(i)]\big)\leq {\rm e}^{-c' i n^a}
	\]
	for some constant \(c'>0\). Therefore, a union bound yields
	\begin{equation*}
		\begin{aligned}
			P(U \text{ is not } a\text{-regular})\leq n^{1-a}{\rm e}^{-c' n^a} \leq {\rm e}^{-c n^a},
		\end{aligned}
	\end{equation*}
	for some \(c<c'\) and \(n\) large enough as \(a<1\); thus proving~(i).
    
	To prove~(ii), we first derive bounds on the empirical distribution function \(F\) of an \(a\)-regular collection \(u=(u_1,\dots,u_n)\). Namely, we have on the event of \(a\)-regularity, for any \(t\in(0,1)\) and \(n\geq 2\),
	\begin{equation*}\label{eq:empiricalDistr}
		\begin{aligned}
			n F(t) 
			& 
				= \sum_{i=1}^n \1\{u_i\leq t\} \geq \sum_{i=1}^{\lfloor n^{1-a}\rfloor}N^a_n(i-1)\1\big\{i-1<\lfloor n^{1-a}\rfloor t\leq i\big\} = N^a_n(\lfloor t \lfloor n^{1-a}\rfloor\rfloor) 
			\\ &
				\geq \frac{n \lfloor t\lfloor n^{1-a}\rfloor\rfloor}{2\lfloor n^{1-a}\rfloor}\geq \frac{n}{2}\Big(t-\frac{1}{\lfloor n^{1-a}\rfloor}\Big)\geq \frac{n}{2}(t-2n^{a-1}).	
		\end{aligned}
	\end{equation*}
	Now observe, that the function \(G(t)=(t-2n^{a-1})/2\) induces the finite measure \(\d t /2\) on \((2n^{a-1},1)\), which is non-empty as \(n>2^{1/(1-a)}\). This particularly implies for all non-increasing functions \(\varphi(t)\) and an \(a\)-regular realisation \(u=(u_1,\dots,u_n)\)
	\begin{equation*}
		\begin{aligned}
			\sum_{j=1}^n \varphi(u_j) 
			& 
				\geq n \int_{0}^{\varphi(2n^{a-1})} \d x \, F(\varphi^{-1}(x)) \geq \frac{n}{2}\int_{2n^{a-1}}^1 \d t \, \varphi(t),
		\end{aligned}
	\end{equation*}
	where \(\varphi^{-1}\) denotes the generalised inverse of \(\varphi\). This concludes the proof. 
\end{proof}

	We have now everything in place in order to prove Theorem~\ref{thm:WDRCM}.
		
	\begin{proof}[Proof of Theorem~\ref{thm:WDRCM}, subject to Theorem~\ref{thm:gWDRCM}.] 
	Let us start with the proof of Part~(i) and assume \(\zeta<0\).
 Note that Theorem~\ref{thm:gWDRCM} Part~(i) (that is independently proved below) implies 
	\[
		\limsup_{r\to\infty}\frac{\log \P_1(\cL(r,1))}{\log r}\leq {d\zeta}. 
	\]
	As the WDRCM is monotone and takes place in the independent setting, we have \(\wlc>0\) by Theorem~\ref{thm:equiStatement}. Let us assume that we had already proved the corresponding lower bound 
	\begin{equation}\label{eq:lowerBoundPL}
		\limsup_{r\to\infty}\frac{\log\P_1(\cL(r,1))}{\log r} \ge d\zeta.
	\end{equation}
	Then, for all \(\lambda<\wlc\) by Theorem~\ref{thm:gWDRCM} and Lemma~\ref{lem:PropL}
	\[
		\limsup_{r\to\infty}\frac{\log\P_\lambda(\cE(r))}{\log\P_1(\cL(r,1))}\leq 1,
	\]
	where \(\cE(r)\) again denotes the annulus-crossing event as introduced in~\eqref{eq:eventAnnulusCross}. As \(\P_1(\cL(r,1))\asymp\P_\lambda(\cL(r,3))\leq \P_\lambda(\cE(r))\) by Lemma~\ref{lem:PropL}, we find the corresponding lower bound for the \(\limsup\) of the long edge-probability and have thus established the claimed polynomial decay of the annulus-crossing probability. It therefore remains to verify the lower bound~\eqref{eq:lowerBoundPL}. To this end, we fix the intensity \(\lambda=1\) and consider the event 
	\[
		\bar{\cL}(r) = \big\{\exists \x\sim\y\colon x\in B(r), y\in B(3r)\setminus B(2r)\big\} \subset \cL(r,1).
	\] 
	We hence aim to show that for all \(\mu<\zeta\) arbitrarily close to \(\zeta\) and sufficiently large \(r\), we have
    \[
        \log\P_1(\bar{\cL}(r)) \geq  d\mu\log r.
    \]
    Fix such \(\mu<\zeta\). We provide the claimed lower bound for the probability of $\bar{\cL}(r)$ by only considering connections between two vertices with marks greater than $r^{-d+d\mu}$. Recall the integral
    \[
        I(\mu, r)= r^{2d}\int_{r^{-d+d\mu}}^1 \d u  \int_{r^{-d+d\mu}}^1  \d v \, \varphi(u,v,r^d).
     \] 
   By our definition of $\psi$ in~\eqref{eq:psi} and our assumption on \(\mu\), we have 
   \[
       \limsup_{r\to\infty} \frac {\log I(\mu,r)}{d\log r}= 2+\psi(\mu)>\mu.
   \]
   We let $\eps:= (2+\psi(\mu)-\mu)/3>0$ and split the integral according to the domain of integration into
   \[
       I(\mu,r)=I^{(1)}(\mu,r)+ I^{(2)}(\mu,r)+ I^{(3)}(r),
   \]
   where the integral $I^{(1)}$ runs over the domain $\{r^{-d+d\mu}<u<r^{-d+d\eps}, u<v<1\}$, $I^{(2)}$ over the domain $\{r^{-d+d\mu}<v<r^{-d+d\eps}, v<u<1\}$, and $I^{(3)}$ over the domain $\{u,v\in [r^{-d+d\eps},1]\}$. Let us recall here that \(\mu<0\) while \(\varepsilon>0\) to avoid confusion. Hence, by the Laplace principle, for at least one $j\in\{1,2,3\}$, we must have       
   \begin{equation}\label{eq:SubintPsiBound}
   		\limsup_{r\to\infty} \frac {\log I^{(j)}(\mu,r)}{d\log r}= 2+\psi(\mu).
   \end{equation}
   We have to consider the three cases associated to which \(j\) satisfies equation~\eqref{eq:SubintPsiBound}. Let us first assume that it holds true for \(j=1\). We can then fix a sequence \(r_n\) diverging to infinity such that \(I^{(1)}(\mu,r_n)\geq r_n^{d\mu+2d\varepsilon}\). We define
   \begin{equation*}
   		\Phi(u) := r_n^d\int_{r_n^{-d+d\mu}}^1  \d v \,  \varphi(u,v,r_n^d)
   \end{equation*}
   and infer
   \begin{equation} \label{eq:lowerBoundContradict}
   		r_n^{d} \int_{r_n^{-d+d\mu}}^{r_n^{-d+d\varepsilon}}\d u \,\Phi(u) \geq I^{(1)}(\mu,r_n)\geq r_n^{d\mu+2d\varepsilon}.
   \end{equation}
    Hence, there must exists \(u_0\in[r_n^{-d+d\mu},r_n^{-d}]\) such that \(u_0 \Phi(u_0)\geq r_n^{-d+d\mu}\). Indeed, if this was not the case, we would infer
    \begin{equation*}
    	\begin{aligned}
    		I^{(1)}(\mu,r_n) \leq r_n^{d\mu} \int_{r_n^{-d+d\mu}}^{r_n^{-d}} \frac{\d u}{u} + r_n^{d\varepsilon}\Phi(r_n^{-d}) \leq r_n^{d\mu}\log r_n^{-d\mu} + r_n^{d\mu+d\varepsilon},
    	\end{aligned}
    \end{equation*} 
    where we used monotonicity of \(\varphi\) in the vertex mark arguments and the assumed bound on \(u\Phi(u)\). This however contradicts~\eqref{eq:lowerBoundContradict} and thus yields the existence of \(u_0\). If there is more than one such \(u_0\), fix any and observe that the probability of finding a vertex located in \(B(r_n)\) with mark at most \(u_0\) is, by standard Poisson process properties together with \(\lambda=1\),
    \[
    	1- \exp\big(-\omega_d r_n^d u_0\big) \geq \tfrac{\omega_d }{2} r_n^d u_0,
    \]	
    for large enough \(n\) (and thus \(r_n\)). Conditioned on the existence of such a vertex, the number of neighbours it has located in \(B(3r_n)\setminus B(2r_n)\) is a Poisson random variable with parameter at least \(\Phi(u_0)\).
  This is immediate from the Poisson structure of neighbourhoods in the WDRCM~\cite{Lue2022} and monotonicity of \(\varphi\) in the vertex marks arguments. Hence, the probability of forming at least one connection from the given vertex to \(B(3r_n)\setminus B(2r_n)\) is lower bounded by
    \[
    	1-\exp\big(-\Phi(u_0)\big) = \Omega(\Phi(u_0)\wedge 1).
    \]
    Hence, we find a constant \(c>0\), depending on model parameters and the dimension only, such that, for all \(\mu<\zeta\) and sufficiently large \(r_n\), we have 
    \[
    	\P_1(\bar{\cL}(r_n))\geq c r_n^d u_0(\Phi(u_0)\wedge 1) = c\big(r_n^d u_0 \Phi(u_0)\big)\wedge \big(r_n^d u_0\big)\geq c r_n^{d\mu},
    \]
    as desired. Note that the case~\eqref{eq:SubintPsiBound} for \(j=2\) follows from the exact same arguments by reversing the roles of \(u\) and \(v\).  
    
   It thus remains to consider the case where~\eqref{eq:SubintPsiBound} is fulfilled for \(j=3\). To this end, recall \(\varepsilon=(2+\psi(\mu)-\mu)/3 \in(0,1)\) introduced to divide the integration domain, where \(\varepsilon<1\) is a result of only considering \(\mu<\zeta\) that are sufficiently close to \(\zeta\). Again, we want to establish the existence of an edge between a vertex in \(B(r)\) and one in \(B(3r)\setminus B(2r)\). For this, let us denote by \(\scrV_1=\{\x_1=(x_1,u_1),\dots,\x_{N_1}=(x_{N_1},u_{N_1})\}\) the vertices of \(\scrV\) restricted to locations in \(B(r)\) and by \(\scrV_2=\{\y_1=(y_1,v_1),\dots,\y_{N_2}=(y_{N_2},v_{N_2})\}\) those vertices with locations in \(B(3r)\setminus B(2r)\). Here, \(N_1\) and \(N_2\) are Poisson random variables with mean of order \(r^d\). Applying the Chernoff bound for Poisson random variables, there are constants \(c',c>0\), such that
	\begin{equation}\label{eq:chernoffPois}
		\P_{1}(N_1< c' r^d) + \P_1(N_2<c' r^d) \leq {\rm e}^{-c r^d}.
	\end{equation} 
	To easy notation, we treat \(c'=1\) and additionally \(r^d\in\N\) in the following and denote by \(c>0\) and \(C>0\) constants that may change from line to line. Further, we may assume without loss of generality that the vertices in \(\scrV_1\) and \(\scrV_2\) may be indexed according to their lexicographic order. We now make use of the previously introduced regularity of uniform random variables applied to the vertex marks in \(\scrV_1\) and \(\scrV_2\). For the parameter \(a=\varepsilon\), we say that \(\scrV_1\) is \(\varepsilon\)-regular, if \(N_1\geq r^d\) and if the collection of vertex marks \(u_1,\dots,u_{r^d}\) is \(\varepsilon\)-regular. Put differently, \(\scrV_1\) is \(\varepsilon\)-regular, if it contains sufficiently many vertices such the the first \(r^d\) indexed vertex marks form a \(\varepsilon\)-regular collection in the above sense. The same definition applies to \(\scrV_2\) by replacing \(N_1\) with \(N_2\) and the marks \(u_i\) with the marks \(v_i\). Applying Lemma~\ref{lem:a-regular} and~\eqref{eq:chernoffPois} yields
	\begin{equation}\label{eq:ConditionEpsReg}
		\begin{aligned}
			\P_1(\neg\bar{\cL}(r)) 
			& 
			\leq \P_{1}\big(\neg\bar{\cL}(r)\cap \{\scrV_1,\scrV_2 \text{ are }\varepsilon\text{-regular}\}\big)+ {\rm e}^{-c r^d} +2 {\rm e}^{-c r^{d\varepsilon}}			\\ &
			\leq \P_{1}\big(\neg\bar{\cL}(r)\cap \{\scrV_1,\scrV_2 \text{ are }\varepsilon\text{-regular}\}\big)+ {\rm e}^{-c r^{d\varepsilon}}, 
		\end{aligned}
	\end{equation}
	for all large enough \(r\), where the last step follows from \(\varepsilon<1\). It remains to bound the probability on the right-hand side. To this end, let us write \(\mathcal{A}=\{\scrV_1,\scrV_2 \text{ are }\varepsilon\text{-regular}\}\). Note that the maximal distance between any two vertices in \(\scrV_1\) and \(\scrV_2\) is bounded by \(6r\). As \(\varphi\) is non-increasing in the third argument, we infer	\begin{equation}\label{eq:longEdgeProofCalc}
		\begin{aligned}
			\P_1(\neg \bar{\cL}(r)\cap\mathcal{A}) 
			& 
			\leq \P_1\Big(\bigcap_{i,j=1}^{r^d}\{\x_i\not\sim\y_j\}\cap \mathcal{A}\Big) 
			\leq \E_1\Big[\1_{\mathcal{A}} \prod_{\substack{i,j=1}}^{r^d} {\rm e}^{-\varphi\big(u_i,v_j,(6r)^d\big)}\Big] 
			\\ & 
			 =\E_1\Big[\1_{\mathcal{A}} \ \exp\Big(-\sum_{i,j=1}^{r^d} \varphi\big(u_i,v_j,(6r)^d\big)\Big)\Big] 
			\\ &
			\leq \P_1(\mathcal{A})\exp\Big(-C r^{2d}\int\limits_{2r^{-d+d\varepsilon}}^{1} \d u \int\limits_{2r^{-d+d\varepsilon}}^{1} \d v \ \varphi\big(u,v,(6r)^d\big) \Big)
			\\ &
			\leq \P_1(\mathcal{A})\exp\big(-C I^{(3)}(r)\big),
		\end{aligned}
	\end{equation}
	where we used Lemma~\ref{lem:a-regular} in the second to last step. To see the last step, recall that \(I^{(3)}(r)\) is of polynomial order with exponent \(2+\psi(\mu)\) by assumption. Therefore, \(I^{(3)}(r)\asymp I^{(3)}(c r)\) for all constants \(c>0\). We ultimately infer, 
	\[
		\begin{aligned}
			\P_1(\bar{\cL}(r,1)) \geq \P_1(\bar{\cL}(r)\cap \mathcal{A}) \ge \P_1(\mathcal{A})\big( 1- \exp(-C I^{(3)}(r))\big)
		\end{aligned} 
    \]
    for sufficiently large \(r\). By taking the logarithms on both sides, using our assumption on \(I^{(3)}\), as well as \(\P_1(\mathcal{A})>1/2\) for large \(r\), we obtain the desired result. This finishes the proof of Part~(i). 

	In order to prove Part~(ii), we only have to show the claimed tail for \(\P_1(\cL(r,1))\) as the other statements are then immediate consequences of Theorem~\ref{thm:equiStatement}. Recall that we now have \(\zeta>0\). The proof works with similar argumentation as the one used in the \(j=3\) case above. Indeed, fix \(0<\mu<\zeta\leq 1\) but arbitrarily close to \(\zeta\) and replace the \(\varepsilon\)-regularity in the event \(\mathcal{A}\) by \(\mu\)-regularity. Then, for any \(\varepsilon>0\) small enough such that \(d(2+\psi(\mu))-\varepsilon>0\), we have for large enough \(r\) that \(I(\mu,r)\geq r^{d(2+\psi(\mu))-\varepsilon}\). We infer with the same calculations performed in~\eqref{eq:longEdgeProofCalc} for large enough \(r\)
	\begin{equation*}
		\begin{aligned}
			\P_1(\neg \bar{\cL}(r)\cap\mathcal{A}) 
			& 
			\leq \P_1(\mathcal{A})\exp\Big(-C r^{2d} \int\limits_{2r^{d\mu-d}}^{1} \d u \int\limits_{2r^{d\mu-d}}^{1} \d v \, \varphi\big(u,v,(6r)^d\big) \Big)
			\\ & 
			\leq\P_1(\mathcal{A}) \exp\big(-C r^{d(2+\psi(\mu))-\varepsilon}\big).
		\end{aligned}
	\end{equation*}
	As \(\varepsilon\) can be chosen arbitrarily small, we can find for each \(\mu<\zeta\) some \(R\geq 1\) and a constant \(C>0\) such that for all \(r>R\), by combining the final calculation with~\eqref{eq:ConditionEpsReg} (but now with respect to \(\mu\)-regularity),
	\[
		\P_1(\cL(r,1))\geq \P_1(\bar{\cL}(r))\geq 1 - \exp(-C r^{d\mu}).
	\]
	This proves the claimed tail and thus Theorem~\ref{thm:WDRCM}. 
	\end{proof}
 
 	We continue with the proof of Proposition~\ref{prop:deffForWDRCM} providing the values of \(\zeta\) for the interpolation model, defined in~\eqref{eq:rho} and~\eqref{eq:interpolKernel}.
	
	\begin{proof}[Proof of Proposition~\ref{prop:deffForWDRCM}.]
		Recall that the connection probability for two given vertices \(\x=(x,u_x)\) and \(\y=(y,u_y)\) is given by \(1\wedge (g(u_x,u_y) |x-y|^d)^{-\delta}\), with \(g(u_x,u_y)=(u_x\wedge u_y)^\gamma (u_x\vee u_y)^\alpha \), in the interpolation model. We aim to derive the values of \(\zeta\). To this end, we calculate the defining integral of \(\psi(\mu)\) in~\eqref{eq:psi}, namely $I(\mu,r)/r^{2d}$. By the symmetry of $g$, this integral is twice the value obtained on the restricted domain in which $u_x\le u_y$. After a change of variables with $t_x=-\log(u_x)/d\log r$ and $t_y=-\log(u_y)/d\log r$, we obtain
		\[
        	\begin{aligned}
        		\frac {I(\mu,r)}{2r^{2d}}
        		&= 
        			(d\log r)^2 \int_R \big(r^{d(-\delta+(\delta \gamma-1) t_x + (\delta \alpha-1)t_y)}\wedge r^{d (-t_x-t_y)}\big) \, \d t_x \d t_y
        		\\ &
        			= (d\log r)^2 \int_R r^{d \chi(t_x,t_y)} \, \d t_x \d t_y,
        	\end{aligned}
        \]
        where the integral runs over the domain $R:=\{(t_x,t_y)\in \R^2, 0\le t_y\le t_x\le 1-\mu\},$ and 
        \[
        	\chi(t_x,t_y):= \min\big\{-\delta+(\delta \gamma-1) t_x + (\delta \alpha-1)t_y, -t_x-t_y\big\}.
        \]
        Note that $\chi(t_x,t_y)=-t_x-t_y$ on the domain $\gamma t_x+\alpha t_y\ge 1$, corresponding to the connection probability between the vertices $\x$ and $\y$ being equal to 1.

        We first consider the case $\mu\ge 1-1/(\alpha+\gamma)$, which yields $\chi(t_x,t_y)=-\delta+(\delta \gamma-1) t_x + (\delta \alpha-1)t_y$ in the whole domain $R$. If we further suppose that the terms $\delta \alpha-1, \delta(\alpha+\gamma)-2$ and $\delta \gamma-1$ are all nonzero, then the integral is easily computed to be equal to:
        \[
        	\frac {I(\mu,r)}{2r^{2d}} =\frac {r^{-d\delta}}{\delta \alpha-1}\left( \frac {r^{d(\delta(\alpha+\gamma)-2)(1-\mu)}-1}{\delta(\alpha+\gamma)-2} - \frac {r^{d(\delta\gamma-1)(1-\mu)}-1}{\delta\gamma-1}\right).
        \]
        The largest exponent in $r$ yields the leading term, so we find
        \[
        	\frac {I(\mu,r)}{r^{2d}}\asymp r^{d\psi(\mu)},
        \]
        where
        \[
            \psi(\mu) =\max\big\{-\delta, -\delta+ ((\alpha+\gamma)\delta-2)(1-\mu),-\delta+(\delta \gamma-1)(1-\mu)\big\}.
        \]
        This formula for $\psi(\mu)$ remains true when either of the three denominators $\delta \alpha-1, \delta(\alpha+\gamma)-2$ or $\delta \gamma-1$ 
        is zero. Indeed, in those cases the asymptotic for $I(\mu,r)$ also possibly contains a logarithmic correction, but this does not change the exponent $\psi(\mu)$. It can also be observed that we have
        \begin{equation}\label{formula_psi_chi}
            \psi(\mu)= \sup\big\{\chi(t_x,t_y)\colon (t_x,t_y)\in R\big\},
        \end{equation}
        and our computations above should also provide a good explanation of this equality.

        The case $\mu<1-1/(\alpha+\gamma)$ requires slightly more involved computations but~\eqref{formula_psi_chi} remains true and provides a quick way to compute the exponent $\psi$. If $1-1/\gamma<\mu< 1-1/(\alpha+\gamma)$, we obtain
        \[
        	\begin{aligned}
            	\psi(\mu)
            	&=
            		\max\left\{\chi(0,0),\chi\big(\tfrac{1}{\alpha+\gamma},\tfrac{1}{\alpha+\gamma}\big),\chi(1-\mu,0),\chi\big(1-\mu,\tfrac {1-\gamma(1-\mu)}{\alpha}\big)\right\}
           		\\&
           			=\max\left\{-\delta, -\tfrac{2}{\alpha+\gamma},-\delta+(\delta \gamma-1)(1-\mu), -\tfrac{1}{\alpha} + (\tfrac{\gamma}{\alpha}-1)(1-\mu)\right\}.
        	\end{aligned}
        \] 
        Finally, if $\mu<1-1/\gamma$, we obtain
        \[
        	\begin{aligned}
            	\psi(\mu)
            	&=
            		\max\left\{\chi(0,0),\chi\big(\tfrac{1}{\alpha+\gamma},\tfrac{1}{\alpha+\gamma}\big),\chi\big(\tfrac{1}{\gamma},0\big)\right\}=\max\left\{-\delta,-\tfrac{2}{\alpha+\gamma}, -\tfrac{1}{\gamma}\right\}.
        	\end{aligned}
        \]
        
        We conclude by summarising our results and computing the value of $\zeta$ depending on the parameters $\alpha,\gamma$ and $\delta$. In each case, $\psi$ is continuous in $\mu$ and $\zeta$ is the unique solution in $(-\infty,1]$ of the equation $\zeta=2+\psi(\zeta)$. We distinguish the various cases as shown in Figure~\ref{fig:Deff}.
        
        \begin{description}
        	\item[Case] \(\boldsymbol{\gamma<1/\delta}\) \textbf{and} \(\boldsymbol{\alpha+\gamma<2/\delta.}\)
        		In this case, we have $\psi(\mu)= -\delta$ for all $\mu$. Further solving $\zeta=2+\psi(\zeta)$ yields $\zeta=2-\delta$, and therefore $\zeta>0$ if $\delta<2$, $\zeta=0$ if $\delta=2$, and $\zeta<0$ if $\delta>2$.
        	\item[Case] \(\boldsymbol{\gamma>1/\delta}\) \textbf{and} \(\boldsymbol{\alpha<1/\delta.}\) 
        		Here, we have
        		\begin{equation*}
					\begin{aligned}
						\psi(\mu)=
						\begin{cases}
							-\frac{1}{\gamma}, & \text{ if } \mu<1-\frac 1 {\gamma}, 
							\\
							-\delta+(\delta\gamma-1)(1-\mu), & \text{ if } \mu>1-\frac{1}{\gamma}.
						\end{cases}
					\end{aligned}
				\end{equation*}
        		Solving $\zeta=2+\psi(\zeta)$, yields $\zeta=1-(\delta-1)/\gamma\delta$, which is always larger than $1-1/{\gamma}$. Further, $\zeta>0$ if $\gamma>1-1/\delta,$ $\zeta=0$ if $\gamma=1-1/\delta$, and $\zeta<0$ if $\gamma<1-1/\delta$.
			\item[Case] \(\boldsymbol{1/\delta<\alpha<\gamma}.\) 
				We infer in this case 
				\begin{equation*}
					\begin{aligned}
						\psi(\mu)=
						\begin{cases}
							-\frac{1}{\gamma}, & \text{ if } \mu<1-\frac 1 {\gamma}, 
							\\
                    		-\frac{1}{\alpha} +\left(\frac{\gamma}{\alpha}-1\right)(1-\mu),  & \text{ if } 1-\frac 1 {\gamma}<\mu <1-\frac 1 {\alpha+\gamma},
                    		\\
							-\delta+(\delta(\alpha+\gamma)-2)(1-\mu), & \text{ if } 1-\frac 1 {\alpha+\gamma}<\mu.
						\end{cases}
					\end{aligned}
				\end{equation*}
        		Note that the solution of $\zeta=2+\psi(\zeta)$ is always larger than $1-1/\gamma$. We obtain, 
        		\[
        			\zeta =
        			\begin{cases}
        				\frac{\alpha+\gamma-1}{\gamma}<0, & \text{ if } \alpha+\gamma<1,
        				\\
        				1-\frac{1}{\alpha+\gamma}=0, & \text{ if } \alpha+\gamma=1,
						\\
						\frac{\delta(\alpha+\gamma-1)}{\delta(\alpha+\gamma)-1}>0, & \text{ if } \alpha+\gamma>1.
        			\end{cases}
        		\]
			\item[Case] \(\boldsymbol{\gamma<\alpha}\) \textbf{and} \(\boldsymbol{\alpha+\gamma>2/\delta.}\)
				In this final case, we obtain
				\begin{equation*}
					\begin{aligned}
						\psi(\mu)=
						\begin{cases}
							-\frac{2}{\alpha+\gamma},  & \text{ if } \mu <1-\frac 1 {\alpha+\gamma},
						\\
							-\delta+(\delta(\alpha+\gamma)-2)(1-\mu), & \text{ if } 1-\frac 1 {\alpha+\gamma}<\mu.
						\end{cases}
					\end{aligned}
				\end{equation*}
				Again solving the usual equation, we infer 
				\[
					\zeta =
					\begin{cases}
						\frac {2(\alpha+\gamma-1)}{\alpha+\gamma}<0, & \text{ if } \alpha+\gamma<1, 
						\\
						1-\frac 1{\alpha+\gamma}=0, & \text{ if } \alpha+\gamma=1,
						\\
						\frac{\delta(\alpha+\gamma-1)}{\delta(\alpha+\gamma)-1}>0, & \text{ if } \alpha+\gamma>1.
					\end{cases}
				\]
        \end{description}
    Reorganising these results according to the sign of $\zeta$, we obtain the expression of $\zeta$ provided in the statement of Proposition~\ref{prop:deffForWDRCM}.
    \end{proof}

    Finally, we give the proof of Proposition~\ref{prop:zeta=0}, dealing with the \(\zeta=0\) phase of the interpolation model.
	
	\begin{proof}[Proof of Proposition~\ref{prop:zeta=0}.]
		Consider first \(\delta=2\). Since the interpolation model naturally dominates long-range percolation, we have \(\P_1(\cL(r,1))\asymp 1\) by the argument mentioned in Remark~\ref{rem:thmWDRCM}~(b). For the remaining cases, observe first that, with probability \(1-e^{\omega_d}>0\), there exists a vertex in \(B(r)\) with mark at most \(r^{-d}\). Given such a vertex, its neighbours in \(B(2r)\setminus B(r)\) with marks at least \(r^{-d}\) form a Poisson process. In the case \(\gamma=1-1/\delta\), we may assume \(\alpha=0\) and observe that the mean number of such neighbours is bounded from below by order \(r^{d}r^{-d\delta(1-\gamma)}\asymp 1\) and for \(\gamma<1-1/\gamma\) and \(\alpha=1-\gamma\), the same mean is bounded by
		\[
			r^d r^{-d\delta(1-\gamma)} \int_{r^{-d}}^{1} \d u \ u^{-\delta(1-\gamma)} \asymp 1.
		\]
		In both cases, with constant probability, there exists a vertex of mark at most \(r^{-d}\) that is incident to a long edge. This concludes the proof. 
	\end{proof}

	
	\subsubsection*{Proofs for the generalised WDRCM of Section~\ref{sec:genWDRCM}}
	We turn now to the proofs for the generalised version of the WDRCM introduced in Section~\ref{sec:genWDRCM}. Recall the framework and assumptions explained around~\eqref{eq:varphi} as well as the generalisations of \(\psi_\lambda\) and \(\zeta_\lambda\). 
	
	\begin{proof}[Proof of Theorem~\ref{thm:gWDRCM}.] Let us start with the proof of Part~(i). Thus we assume \(\zeta_\lambda<0\) for some \(\lambda>0\). Choose some \(\zeta_\lambda/2>\mu>\zeta_\lambda\) arbitrarily close to \(\zeta_\lambda\). We aim for an appropriate bound for the probability of \(\cL(r,3)\). To this end, define
		\[
			\widetilde{\cL}(r):=\widetilde{\cL}(r,\mu) =\big\{\exists \ \x\sim \y\colon x\in B(r)\text{ and }y\in B(2r)^\textsf{c}, u_x\wedge u_y\geq |y|^{-d(1+\mu)}\big\}
		\]
		and note that
		\[
			\cL(r,3)\subset \widetilde{\cL}(r)\cup\big\{\exists \y\in\scrV\colon y\in B(r)^\textsc{c}, u_y<|y|^{-d(1+\mu)}\big\} \cup \big\{\exists \x\in\scrV\colon x\in B(r), u_x<r^{-d(1+\mu)} \big\}.
		\]
		By standard properties of the Poisson point process, the latter two events on the right-hand side have probabilities of order
		\begin{equation*}
			\lambda \int_{|y|>r}\d y \, |y|^{-d(1+\mu)} = \tfrac{\omega_d}{|\mu|} \lambda r^{d\mu} \leq \tfrac{\omega_d}{|\mu|\wedge 1}\lambda r^{d\mu}\  \text{ and } \ \lambda\int_{|x|<r} \d x \, r^{d(1+\mu)} = \omega_d \lambda r^{d\mu}\leq \tfrac{\omega_d}{|\mu|\wedge 1}\lambda r^{d\mu}.
		\end{equation*}	
		Let us bound next the probability of \(\widetilde{\cL}(r)\). Applying Mecke's equation~\cite{LastPenrose2017} and~\eqref{eq:varphi}, we obtain 
		\begin{equation*}
			\begin{aligned}
				\P_\lambda(\widetilde{\cL}(r)) 
				& 
				\leq \E_\lambda\Big[\sum_{\x: x\in B(r)}\sum_{\y: y\in B(2r)^\textsc{c}} \1_{\{u_x\wedge u_y >|y|^{-d(1+\mu)}\}}\mathbf{p}(\x,\y,\scrV\setminus\{\x,\y\})\Big] 
				\\ &
				= \lambda^2 \int\limits_{|y|>2r}\d y \int\limits_{|x|<r}\d x \int\limits_{|y|^{-d+d\mu}}^1 \d u_x \int\limits_{|y|^{-d+d\mu}}^1 \d u_y \ \E_\lambda \mathbf{p}(\x,\y,\scrV)
				\\ &
				= \lambda^2 \int\limits_{|y|>2r}\d y \int\limits_{|x|<r}\d x \int\limits_{|y|^{-d+d\mu}}^1 \d u_x \int\limits_{|y|^{-d+d\mu}}^1 \d u_y \ \varphi_\lambda(u_x,u_y,|x-y|^d).
			\end{aligned}
		\end{equation*}
		Now observe that \(|x|\) and \(|x-y|\) are of the same order for large \(r\). Hence, we find a constant \(C>1\) such that for \(\varepsilon>0\) and all sufficiently large \(r\), by definition of \(\psi_\lambda(\mu)\),
		\begin{equation*}
			\begin{aligned}
				\P_\lambda(\widetilde{\cL}(r)) 
				& 
				\leq C \lambda^2 \omega_d r^d \int\limits_{|x|>r}\d x \int\limits_{|x|^{-d+d\mu}}^1 \d u_x \int\limits_{|x|^{-d+d\mu}}^1 \d u_y \ \varphi_\lambda(u_x,u_y,|x|^d)
				\\ &
				\leq C \lambda^2 \omega_d r^d \int\limits_{|x|>r} \d x \, |x|^{d(\psi(\mu)+\varepsilon)}
				\\ & 
				\leq C \omega_d^2 \lambda^2 r^{d(2+\psi(\mu)+\varepsilon)}.
			\end{aligned}
		\end{equation*}
		As \(\varepsilon\) can be chosen arbitrarily small, we can find for all \(\zeta_\lambda/2>\mu>\zeta_\lambda\) some \(R>1\) such that for all \(r>R\)
		\begin{equation}\label{eq:upperBoundPL}
			\P_\lambda(\cL(r,2)) \leq (\lambda\vee \lambda^2)C\big(1 +\tfrac{2}{|\zeta_\lambda|\wedge 1}\big) \, r^{d\mu}.
		\end{equation}
		This proves Part~(i). In order to prove Part~(ii), observe that in case of \(\overline{\zeta}_{\lambda'}<0\), we can take the supremum \(\sup_{\lambda<\lambda'}\) on both sides of~\eqref{eq:upperBoundPL} and note that the \(\lambda\) dependent constant remains finite. This yields the claimed decay of the long event and particularly implies Property~\ref{G:noLongEdgeUnif}\(_{\lambda'}\). As further~\ref{G:mixing}\(_{\lambda'}\) by assumption, \(\wlc>0\) by Theorem~\ref{thm:main}. This proves Part~(ii). Finally, Part~(iii) is a consequence of Part~(i) and Theorem~\ref{thm:Diamter}. 
	\end{proof}

	Before we turn to the proof of Corollary~\ref{corol:BoolInterferences}, recall the soft Boolean model with local interferences, introduced in Section~\ref{sec:BoolInterferences}.
	
	\begin{proof}[Proof of Corollary~\ref{corol:BoolInterferences}]
		We start by proving that the soft Boolean model with local interferences is a generalised WDRCM. That is, we show that the connection rule~\eqref{eq:BoolInferences} satisfies the integration condition~\eqref{eq:varphi}. Firstly, we observe by the translation invariance of the Poisson point process that \(\mathscr{N}^\beta((x,u),\scrV)\) and \(\mathscr{N}^\beta((o,u),\scrV)\) have the same  for all given coordinates \((x,u)\in\R^d\times(0,1)\). Secondly, since \(\mathscr{N}^\beta((o,u),\scrV)\) is Poisson distributed with parameter \(\lambda \omega_d u^{-\beta}\), it is straightforward to deduce
	\begin{equation} \label{eq:inversePois}
		\E_\lambda \Big[\frac{1}{1+\mathscr{N}^{\beta}((o,u),\scrV)}\Big] = \frac{1-e^{-\lambda \omega_d u^{-\beta}}}{\lambda \omega_d u^{-\beta}} \asymp u^{\beta}.
	\end{equation}    
		Therefore, for two given vertices \(\x=(x,u_x)\)	 and \(\y=(y,u_y)\) with \(|x-y|=r\), Condition~\eqref{eq:varphi} is satisfied with
	\[
		\varphi(u_x,u_y,r^d) = \1_{\{u_x<u_y\}} \frac{(1-e^{-\lambda \omega_d u_x^{-\beta}})(1\wedge u_x^{-\gamma\delta}r^{-d\delta})}{\lambda \omega_d u_x^{-\beta}}+\1_{\{u_x\geq u_y\}} \frac{(1-e^{-\lambda \omega_d u_y^{-\beta}})(1\wedge u_y^{-\gamma\delta}r^{-d\delta})}{\lambda \omega_d u_y^{-\beta}}.
	\]
		We show next, that the model has~\ref{G:P_mixing}\(_\lambda^\xi\) with \(\xi = d(1-1/\beta)\) for all \(\lambda>0\). We note initially that the realisations of the underlying Poisson point process in the two disjoint balls are independent. Moreover, a vertex located in \(B(3r)\) can only interfere an internal edge in \(B(xr,3r)\) subject to \(\cG(xr,r)\) if the sphere of interference of said vertex intersects \(B(xr,3r)\). As the two balls are at distance \(cr\), this requires a mark smaller than \((cr)^{-d/\beta}\) (by definition of \(\mathscr{N}^\beta\)). Thus, by symmetry and Markov's inequality, 
	\begin{equation*}
		\begin{aligned}
			\operatorname{Cov}_\lambda(\1_{\cG(r)}, \1_{\cG(xr,r)}) \leq 2\P_\lambda(\exists \x\in B(r)\cap\scrV\colon u_x<(cr)^{-1/\beta}) \leq 2\omega_d \lambda (cr)^{d(1-1/\beta)}.
		\end{aligned}
	\end{equation*}
		Note that \(1-1/\beta-1<0\) as \(\beta<1\). To calculate \(\zeta\) and identify the phase it is negative in, we calculate by making use of~\eqref{eq:inversePois}
	\begin{equation*}
		\begin{aligned}
			n^{-d\delta}\int\limits_{n^{-d+d\mu}}^1 \d u \, u^{-\gamma\delta+\beta} \asymp n^{-d\delta} \vee n^{d(-\delta-(1-\mu)(1-\gamma\delta+\beta))}.
		\end{aligned}
	\end{equation*}  
		Hence, \(\zeta=2-\delta\) and \(\zeta<0\) for \(\delta>2\) if \(\gamma<(1+\beta)/\delta\), as well as 
		\[
			\zeta = \frac{1-\beta-\delta(1-\gamma)}{\gamma\delta-\beta}
		\] 
		otherwise. In the latter case, we have \(\zeta<0\) if \(\delta>2\) and \(\gamma<(\delta+\beta-1)/\delta\). The proof concludes by applying Theorem~\ref{thm:gWDRCM}.	
	\end{proof}

\subsubsection*{Proofs of Section~\ref{sec:correlatedExamples}}
	Next, we prove the results about the examples in Section~\ref{sec:correlatedExamples}. We start with Corollary~\ref{corol:ellipsis} stated in Section~\ref{sec:ellipsis} about ellipses percolation. We invite the readers to familiarise themselves with the model again.
	
	\begin{proof}[Proof of Corollary~\ref{corol:ellipsis}]
		As \(\lc=0\) implies \(\wlc=0\), the second part of the statement is a direct consequence of~\cite[Theorem~1.2]{Teixeira_Ungaretti_2017} where \(\lc=0\) is shown for \(\gamma>1\). For the first part, we can make use of~\cite[Lemma 2]{Hilario-Ungaretti-2021}, which	states
		\[
			\P_\lambda\big(\exists x\sim y\colon x\in B(r), y\in B(2r)^\textsc{c}\big)\leq 1-\exp\big(-\lambda c r^{2(1-1/\gamma)}\big)
		\]
		for some constant \(c>0\) and \(r\geq 1\). As \(\cL(r,2)\) implies the event on the left-hand side and the right-hand side converges to zero for \(\gamma<1\), we have \(\wlc>0\) by Theorem~\ref{thm:equiStatement}, and the claimed decay of the annulus-crossing probability by Theorem~\ref{thm:Diamter}. 
	\end{proof}
	
	We now turn to the proof of Proposition~\ref{prop:Cox} showing that Cox processes introduced in Section~\ref{sec:Cox} can be suitable candidates for correlated underlying point clouds. 
	
	\begin{proof}[Proof of Proposition~\ref{prop:Cox}]
		The fact that the Conditions~\ref{G:Point_process} and~\ref{G:Translation} are satisfied is clear from the construction. First, we verify mixing or more precisely \ref{G:P_mixing}\(_\lambda^\xi\) for all \(\lambda>0\) for \(\xi=2-\beta\). To this end, note that we can write
		\begin{equation*}
			\P_\lambda\big(\mathcal{G}(r)\cap \mathcal{G}(xr,r)\big)=\E[f_r(\Lambda_{B(r)})g_r(\Lambda_{B(xr,r)})],
		\end{equation*}
		where $f_r(\Lambda_{B(r)})=\P_\lambda(\mathcal{G}(r)|\Lambda)$ and $g_r(\Lambda_{B(xr,r)})=\P_\lambda(\mathcal{G}(xr,r)|\Lambda)$, and where we used the independence of the marked Poisson points as well as the fact that the edge-drawing is assumed to be local by~\eqref{eq:CoxLocFinite}. Here, the notation $\Lambda_{B(x,r)}$ represents the fact that $\Lambda$ is evaluated only in the area $B(x,r)$. Now, in order to control the correlation coming from the environment, consider the three events
		\begin{equation*}
			\begin{aligned}
				\mathcal{A}_r &=\{\forall Y_i\in B((1+\varepsilon)r)\colon \rho_i<\varepsilon r\}
				\\
				\mathcal{B}_r&=\{\forall Y_i\in B(xr,(1+\varepsilon)r)\colon \rho_i<\varepsilon r\}
				\\
				\mathcal{C}_r&=\big\{\forall Y_i\in \big(B((1+\varepsilon)r)\cup B((1+\varepsilon)r)\big)^{\textsf{c}} \colon \rho_i< (|Y_i|\wedge |Y_i-yr|)-r\big\},
			\end{aligned}	
		\end{equation*}
		with $\varepsilon=(|x|-2)/2$. Note that, since $\Y=((Y_i,\rho_i):i\in\N)$ is an i.i.d.~marked Poisson point process, the three events are independent as they refer to points in disjoint regions of $\R^2$. Moreover, under the event $\mathcal{A}_r\cap \mathcal{B}_r\cap \mathcal{C}_r$, the random variables $f_r(\Lambda_{B(r)})$ and $g_r(\Lambda_{B(xr,r)})$ are independent. Hence, by translation invariance, Markov's inequality, and Mecke's formula, 
		\begin{equation*}
			\big|\operatorname{Cov}_\lambda\big(\1_{\mathcal{G}(r)}, \1_{\mathcal{G}(xr,r)}\big)\big|\leq c_1(\P_\lambda(\neg\mathcal{A}_r)+\P_\lambda(\neg\mathcal{B}_r)+\P_\lambda(\neg\mathcal{C}_r))\leq c_2r^{2-\beta}, 
		\end{equation*}
		as desired. For the Condition~\ref{G:locallyFinite}, we use Mecke's formula and Markov's inequality to infer
	\begin{equation*}
    	\begin{aligned}
        	\P\lambda(\exists \x\colon \deg(\x)=\infty)
        	&
        		=\lim_{n\uparrow\infty}\P_\lambda(\exists \x\in B(n)\colon \deg(\x)=\infty)
			\\&
				\leq \lim_{n\uparrow\infty}\E_\lambda\Big[\sum_{ X_i\in \X\cap B(n)}\1\{\deg(\x)=\infty\}\Big]
			\\&
			 	= \lambda\lim_{n\uparrow\infty}\E \Big[\int_{B(n)}\Lambda(\d x)\int_0^1 \d u_x \, \P_\lambda^{x,u_x}(\deg(\x)=\infty|\Lambda)\Big]
			\\&
				\leq \lambda\lim_{n\uparrow\infty}\lim_{m\uparrow\infty}\frac{1}{m}\E\Big[\int_{B(n)}\Lambda(\d x)\int \d u_x \, \E_\lambda^{x,r_x}\Big[\sum_{\X \in \scrV}\1\{\X\sim \x\}\Big|\Lambda\Big]\Big]
			\\&
				= \lambda^2\lim_{n\uparrow\infty}\lim_{m\uparrow\infty}\frac{1}{m}\E \Big[\int_{B(n)}\Lambda(\d x)\int \d u_x \int\Lambda(\d y)\int \d u_y \, \p(\x,\y)\Big]
			\\&
				= \lambda^2\lim_{n\uparrow\infty}\lim_{m\uparrow\infty}\frac{1}{m}\int_{B(n)}\d x\int \d u_x \,\int \d y\int \d u_y \, \p(\x,\y)\E[\ell_x\ell_y].
    	\end{aligned}
	\end{equation*}
	Another application of Mecke's formula as well as the normalisation assumption shows $\E[\ell_x\ell_y]\leq 2$ and thus translation-invariance of $\p$ yields 
	\begin{equation*}
    	\begin{aligned}
			\int_{B(n)}&\d x\int \d u_x \int\d y\int \d u_y \, \p(\x,\y)\E[\ell_x\ell_y]\\
			&\le 2\omega_2 n^2\int \d u_o \int\d x\int \d u_x \, \p((o,r_o),(x,r_x)),
    	\end{aligned}
	\end{equation*}
	which is finite by assumption. Hence, $\P(\exists \x\colon \deg(\x)=\infty)=0$.  Finally, the Condition~\eqref{G:monotone} directly follows from the standard monotone coupling for Poisson point processes once we have conditioned on the environment. 
	\end{proof}
	
	Finally, we prove our result about worm percolation stated in Corollary~\ref{corol:FRI}. 
	
	\begin{proof}[Proof of Corollary~\ref{corol:FRI}]
		Recall percolation of worms introduced in Section~\ref{sec:FRI}. Clearly, the model is translation invariant, locally finite, monotone, and always has Property~\ref{G:noLongEdgeUnif}\(_\lambda\) for all \(\lambda>0\) as we restricted ourselves to nearest-neighbour edges only. It therefore remains to proof mixing. In order to do so, we argue similar as for the Boolean model with local interference and infer
		\begin{equation*}
			\begin{aligned}
				\operatorname{Cov}_\lambda(\1_{\mathcal{G}(r)}, \1_{\mathcal{G}(Bxr,r)}) 
				&
				\leq 2\P_\lambda\big(\exists \x\in B(r)\colon  \text{a worm starting in }x \text{ reaches }B(xr,r)\big) 
				\leq 2\lambda r^d \ell[r,\infty), 
			\end{aligned}
		\end{equation*} 
		which tends to zero as \(r\to\infty\) uniformly in \(\lambda\) whenever \(\ell\) has finite \(d\)-th moment. Thus the model has~\ref{G:mixing}\(_\lambda\) for all \(\lambda\) in this case, and we find \(\wlc>0\) by Theorem~\ref{thm:main}. Let us now assume that the \(d\)-th  moment of \(\ell\) is infinite. We aim to show that at least one worm starting in \(B(r)\) manages to reach \(B(2r)^\textsf{c}\) and an annulus-crossing event thus occurs. To this end, consider the following thinning process. First, we remove each site of \(B(r)\) with probability \(1-\exp(-\nu)\). That is, each remaining site has assigned at least one worm. Exactly one independent worm is started in each of these sites and we keep only those sites whose worms manage to reach \(B(2r)^\textsf{c}\). As the probability of this to happen is decreasing in the distance of the starting site to the target boundary, said probability is bounded from below by \(\P_o^\ell(o\xleftrightarrow[]{} B(2r)^\textsf{c})\). By independence, the number of sites in \(B(r)\) surviving the thinning process is bounded from below by a Binomial random variable with mean of order 
		\begin{equation*}
			\begin{aligned}
				r^d(1-{\rm e}^{-\nu})\P_o^\ell(o\xleftrightarrow[]{} B(2r)^\textsf{c})
				&
				\asymp r^d \sum_{n\geq 2r} \P_o^{\ell=n}\big(\sup_{k\leq n}|S_n|\geq 2r \big)\ell\{n\}
				\asymp r^d \ell[2r,\infty).
			\end{aligned}
		\end{equation*}
		Here, \(S_n\) denotes a simple random walk and \(\P^{\ell=n}_o\) the measure of a random walk, starting in the origin and performing exactly \(n\) steps. The last step follows from the fact that for each fixed \(r\), the probability \(\P_o^{\ell=n}\big(\sup_{k\leq n}|S_n|\geq 2r \big)\) is bounded from zero and tends to one as \(n\to\infty\). As the right-hand side diverges to infinity as \(r\to\infty\), at least one vertex in \(B(r)\) is going to survive the described thinning procedure with probability approaching one. This concludes the proof.
	\end{proof}

\paragraph{Acknowledgement.} 
{BJ and LL gratefully received support by the Leibniz Association within the Leibniz Junior Research Group on \emph{Probabilistic Methods for Dynamic Communication Networks} as part of the Leibniz Competition (grant no.~J105/2020) and by the Deutsche Forschungsgemeinschaft (DFG, German Research Foundation) under Germany's Excellence Strategy -- The Berlin Mathematics Research Center MATH+ (EXC-2046/1, EXC-2046/2, project ID: 390685689) through the project \emph{EF45-3} on \emph{Data Transmission in Dynamical Random Networks}. Moreover, the authors would like to thank an anonymous reviewer whose detailed comments and suggestions on an earlier version by BJ and LL motivated them to join forces with EJ to write this more extensive and improved article that includes many examples some of which suggested in the reviewer's comments.}

\section*{References}
\phantomsection
\addcontentsline{toc}{section}{References}
\renewcommand*{\bibfont}{\footnotesize}
\printbibliography[heading = none]
\end{document}